\documentclass[11pt,leqno]{article}

\usepackage[shortlabels]{enumitem}

\usepackage[all]{xy}				
\usepackage{amsmath}	
\usepackage{amsthm}				
\usepackage{amsfonts}				
\usepackage{bussproofs}		     		
\usepackage{cancel}				
\usepackage{fancyhdr}				
\usepackage{geometry}				
\usepackage{graphicx}				
\usepackage{hyperref}				
\usepackage[utf8]{inputenc}			
\usepackage{lineno}				
\usepackage{mathrsfs}				
\usepackage{MnSymbol}			    	
\usepackage{tikz}				
\usetikzlibrary{tikzmark}
\usepackage{url}				
\usepackage{xcolor}				
\usepackage{prooftree}
\usepackage{proof}
\usepackage[all]{xy}

\usepackage[numbers]{natbib}

\theoremstyle{theorem}
\newtheorem{theorem}{Theorem}[section]

\newtheorem{proposition}[theorem]{Proposition}
\newtheorem{corollary}[theorem]{Corollary}

\theoremstyle{definition}
\newtheorem{definition}[theorem]{Definition}
\newtheorem{remark}[theorem]{Remark}

	\newif\ifSuppressEndOfProva\SuppressEndOfProvafalse

\usepackage{setspace}
\usepackage[shortlabels]{enumitem}

\newcommand{\defin}{\ensuremath{~\stackrel{\text{{\tiny def }}}{=}~}}

\newcommand{\noi}{\noindent}

\newcommand{\cpl}{\text{\it CPL}}

\newcommand{\lfis}{\textit{LFI}s}
\newcommand{\lfi}{\textit{LFI}}

\newcommand{\letfp}{\ensuremath{LET_{F}^+}}

\newcommand{\letf}{$LET_{F}$}
\newcommand{\letk}{\ensuremath{LET_{K}}}
\newcommand{\letj}{\ensuremath{LET_{J}}}

\newcommand{\letkp}{\ensuremath{LET_{K}^+}}

\newcommand{\letfm}{\ensuremath{LET_{F}^-}}

\newcommand{\fde}{\ensuremath{FDE}}
\newcommand{\fdeto}{\ensuremath{FDE^{\to}}}

\newcommand{\bo}{\textsf{b}}
\newcommand{\nei}{\textsf{n}}

\newcommand{\nel}{\textit{N4}}

\newcommand{\lets}{\textit{LET}s}

\newcommand{\imp}{$\rightarrow$}

\newcommand{\cons}{\ensuremath{{\circ}}}
\newcommand{\con}{\ensuremath{{\circ}}}

\newcommand{\toot}{\leftrightarrow}

\newcommand{\wneg}{\ensuremath{\lnot}}
\newcommand{\sneg}{\ensuremath{{\sim}}}
\newcommand{\defi}{\stackrel{\text{\tiny def}}{=} }
\newcommand{\conj}{$\land$}
\newcommand{\disj}{$\lor$}

\newcommand{\mh}{\noindent}

\newcommand{\m}{\vspace{1mm}}
\newcommand{\mm}{\vspace{2mm}}
\newcommand{\mmm}{\vspace{3mm}}
\newcommand{\mmmm}{\vspace{4mm}}
\newcommand{\setl}{\setlength\itemsep{-0.2em}}

\begin{document}

\title{\textbf{On  six-valued logics of evidence and truth expanding Belnap-Dunn four-valued logic} 
\thanks{The definitive version of this article will appear in {\em Studia Logica} as ``From Belnap-Dunn four-valued logic to six-valued logics of evidence and truth''.}
}

\author{Marcelo E. Coniglio$^1$ and Abilio Rodrigues$^2$ \\
[4mm]
$^1$Institute of Philosophy and the Humanities (IFCH), and\\
Centre for Logic, Epistemology and the History of Science (CLE)\\
University of Campinas (UNICAMP) \\
coniglio@unicamp.br\\
$^2$Department of Philosophy\\
Federal University of Minas Gerais (UFMG) \\ abilio.rodrigues@gmail.com \\ }
 
\date{}

\maketitle

\abstract{The main aim of this paper is to introduce the logics of evidence and truth \letkp\ and \letfp\ together with a sound, complete, and decidable  
six-valued deterministic semantics for them. 
These logics extend the logics \letk\ and \letfm\  with rules 
of propagation of classicality, which are  inferences   
that express how the classicality operator \con\ is transmitted from less complex to more complex sentences, and vice-versa. 
The six-valued semantics here proposed  extends the 4 values of Belnap-Dunn logic  with 2 more values that intend to represent   
(positive and negative) reliable information. 
A six-valued non-deterministic semantics for \letk\ is obtained by means of Nmatrices based on swap structures, and the six-valued semantics for \letkp\  is then obtained by 
imposing restrictions on the semantics of \letk. These restrictions correspond exactly to the rules of propagation of classicality that extend  \letk. The logic \letfp\ is obtained as the  implication-free fragment of \letkp. We also show that the 6 values of \letkp\ and \letfp\ define a lattice structure that extends the lattice \textbf{L4} defined by the Belnap-Dunn four-valued logic with the 2  
additional  values mentioned above, intuitively interpreted as positive and negative reliable information. 
 Finally, we also show that \letkp\ is Blok-Pigozzi algebraizable and that its implication-free fragment \letfp\ coincides with the degree-preserving logic of the involutive Stone algebras. }

\maketitle

\section*{Introduction} \label{sec:introduction}

Logics of evidence and truth  (\lets) are paradefinite (i.e.~paraconsistent and paracomplete)  logics that  extend  the logic of first-degree entailment (\fde), also known as Belnap-Dunn four-valued logic   \cite{belnap.1977.how, belnap.1977.useful,  dunn76}, with a unary operator \con\ that
recovers excluded middle and explosion  for sentences in its scope by means of 
the following inferences: 
\begin{enumerate}
\item $\con A, A, \neg A \vdash B$,   \label{prop.exp.con}
\item $\con A \vdash   A \lor  \neg A $. \label{prop.pem.con}  
\end{enumerate}  
\mh    \lets\  stemmed from the logics of formal inconsistency 
    (\lfis, see e.g.~\cite{car.con.2016.book,car.con.mar.2007}), which, in turn, are a further development of the seminal work of da Costa  on paraconsistency (see e.g. \cite{costa1963,costa1974}).     
    In the latter, a sentence $\con A$ means that $A$ is `well-behaved';  in \lfis, that $A$ is `consistent'.\footnote{In da 
    Costa's work, the notation is a bit different, it is written as $A^\con$     instead of $\con A$. 
    The notation $\con A$ has been introduced  by Carnielli and Marcos  in \citep{car.mar.2002.tax}. } 
 \lets, however,  have been conceived with a specific purpose, viz., 
 to formalize the deductive behavior of positive and negative evidence, 
which   can be conclusive or non-conclusive (see e.g. \citep{qletf,letj,rod.ant.lu,letf}). 
    Thus, according to the intended intuitive interpretation,  a sentence  
$\con A $ means that there is conclusive evidence for $A$. It is assumed that sentences for which the evidence 
available is considered conclusive behave classically, and \con\ is called a classicality operator. 
 \lets\  can    also be interpreted as   information-based logics, which are     
     logics suitable for processing
information in the sense of taking a database as a set of premises and drawing
conclusions from these premises in a sensible way. 
 In this case, a
sentence $\con A$  means that the information $A$, positive or negative,
is reliable. 

The inferences \ref{prop.exp.con} and  \ref{prop.pem.con} above added to \fde\ define the logic 
\letfm, which is a sort of minimal logic of evidence and truth.  Kripke models for a  first-order version of \letfm\ 
have been investigated in \citep{rod.ant.lu}.   
The logic \letk, introduced in \citep{kolkata}, extends  the 
logic \fdeto\ with \ref{prop.exp.con} and  \ref{prop.pem.con}. \fdeto\  is \fde\ with a classical implication \citep{hazen, omo.wan.2017}. 
Sound and complete non-deterministic  valuation semantics for sentential \lets\ 
 can be found in \citep{kolkata, letj, letf}, and Kripke-style   semantics in \citep{axioms, handbook.2022}. 
It can be proven that none of the \lets\ studied so far  can be characterized by a single finite logical matrix -- 
that is, they are not  deterministically  finitely-valued. 
  Our main aim here is to obtain a finite-valued non-deterministic semantics for \letk, as well as to introduce the sentential logics \letkp\ and \letfp\ together with a sound, complete, and decidable  
six-valued semantics for them. 
These logics extend, respectively, \letk\ and \letfm\ with  rules 
 of propagation of clasicality, which are  inference rules 
that express how the 
operator \con\ is transmitted from less complex to more complex sentences, and vice-versa. 

\mm

The remainder  of this paper is structured as follows. 
In Section~\ref{sec.lets.fde} we explain the motivation for extending \fde\ with the operator \con\ and the respective 
rules, adding thus two more scenarios to the four scenarios of \fde.\footnote{Parts of Section 
\ref{sec.lets.fde} have already appeared in \citep{rod.ant.lu}. } 
In Section~\ref{sec.letk}, a natural deduction  system for \letk\ is presented, together 
with a sound, complete, and decidable  six-valued semantics based on Nmatrices. 
Section~\ref{sec.letkp} introduces the  logic \letkp\  together 
with a sound, complete, and decidable \textit{deterministic} six-valued semantics.  
The latter is obtained by means of restrictions imposed on the non-deterministic six-valued 
semantics for \letk, and such restrictions correspond exactly to the rules of 
propagation of classicality added to \letk. In this section we also show 
 that \letkp\ is algebraizable in the   sense of  
 Blok and  Pigozzi  \cite{blok:pig:89}. 
  In Section~\ref{sect-lattice}  
 we will show  how  
 the underlying algebra of the  logical matrices of \letkp,   presented as a three-dimensional twist-structure,  
can be  generalized to twist-algebras based on arbitrary Boolean algebras.  
 In this section we also discuss the lattice  \textbf{L6} and the semi-lattice \textbf{A6} as expansions of the 
well known  lattices \textbf{L4} and \textbf{A4}, defined by the the four values of \fde\ in Belnap \cite{belnap.1977.how}. 
 In Section~\ref{sec.letfp}  we turn to the logic  \letfp, and show how the results obtained up to that point 
 can be extended to \letfp. This analysis will raise an interesting relation between \letfp\  and the degree-preserving logic associated with  a variety of algebras known as involutive Stone algebras.

\section{Six scenarios instead of four} 
\label{sec.lets.fde}

Belnap in  \citep{belnap.1977.how} introduced a four-valued semantics for \fde\ designed to 
 represent the information stored in a possibly inconsistent and incomplete database.  
 The semantic values of \fde, represented here by $T_0$, $F_0$, $\bo$, and $\nei$,  
   allow the following four scenarios to be expressed with respect to a given sentence $A$: 

\begin{enumerate}[resume]
 \item[(i)] $v(A) = T_0$: $A$ holds and $\neg A$ does not hold (only positive information $A$); 
 \item[(ii)]  $v(A) = F_0$: $\neg A$ holds and $A$ does not hold (only negative information $A$); 
  \item[(iii)]  $v(A) = \bo$: both $A$ and $\neg A$ hold (contradictory information about $A$);  
  \item[(iv)]  $v(A) = \nei$: neither $A$ nor $\neg A$ holds (no information about $A$).\footnote{Note that the 
   semantics introduced by  
Dunn \citep{dunn76} in terms of subsets of  $\{ 0,1\} $ contains essentially the same idea. Although the interpretation in terms of information was spelled out by 
Belnap in \citep{belnap.1977.how}, the corresponding conceptual and technical ideas were worked out by Dunn as well.  
On  the historical background of \fde\ and its interpretation in terms of information, from Dunn's dissertation \cite{dunn66} to the papers published in the 1970s \cite{belnap.1977.how, belnap.1977.useful,  dunn76},  see Dunn \cite{dunn2019}.} 
\end{enumerate}

As already said,   \lets\   admit an interpretation 
in terms of positive and negative information, and in this case 
 $\con A$  means that the information $A$, positive or negative,
is reliable,  
 and when $\con A$ does not hold, it means 
 that there is no reliable information about $A$.\footnote{For a more detailed discussion of the notions of evidence and 
 information,  see \citep{rod.car.llp}. Concerning the intuitive interpretation of \lets\ in terms of information, see 
\citep{rod.ant.lu}.} 
     As a consequence,  \lets\ are  
    are able to express  six scenarios with respect to a sentence $A$: (i) to (iv) above when $\con A$ does not hold, 
    plus the following two, represented here by the semantic values $T$ and $F$: 

\begin{enumerate}[resume]
 \item[(v)] $v(A) = T$: $\con A$ and $A$ hold: reliable information that $A$ is true;  
 \item[(vi)] $v(A) = F$: $\con A$ and $\neg A$ hold: reliable information that $A$ is false.
 \end{enumerate}

\noi 
\lets\  thus establish a distinction that cannot be established within \fde, 
 for when the values $T_0$  or  $F_0$ are assigned to $A$  in \fde\ this does not specify  
 whether such information is or is not reliable, which is precisely the
difference between scenarios (i) and (v), and (ii) and (vi) above.

\section{The logic \letk}    \label{sec.letk}

We start with  the logic \letk, which extends the logic \fdeto\ with the operator \con\  and 
 the respective  rules. 
\fdeto\  is \fde\  with a classical implication added, and  can also be defined just by adding $A\lor (A\to B)$ to Nelson's logic \nel. 
 From the semantical point of view, \fdeto\ admits a four-valued semantics 
that extends in a natural way the semantics of \fde\ (see \cite[p.~1036]{omo.wan.2017}). 
{As far as we know, the logic \fdeto\ 
appeared for the first time in Pynko \cite[p.~70]{pynko.1999}, under the name $\mathcal{IDM}4$. It 
 appears in  Hazen et al. \cite[p.~165]{hazen} under the name \fdeto, which we adopt here. }

As already mentioned in the Introduction, \letk\ is not (deterministically) finitely-valued.  It admits a Kripke-style semantics \cite{handbook.2022} and a valuation semantics  \citep{kolkata}.\footnote{Valuation semantics are  bivalued and possibly non-deterministic semantics introduced and investigated in the 1970s onwards by da Costa, Loparic, and Alves (see e.g. \cite{alves.1976,costa1974,loparic.1986,loparic.alves,loparic.costa.1984}).     
The underlying idea of valuation semantics is to `mirror' the axioms and rules in terms of the semantic values 1 and 0. 
  On the historical and conceptual features of valuation semantics, see  \cite[Sect.~6]{qletf}.}
 Since \letk\ is paradefinite, the semantic values of literals $p$ and $\neg p$ are totally independent of each other, 
which is precisely the point of the four scenarios of \fde.  

 In what follows we will see a natural deduction system for \letk, 
 together with an adequate valuation semantics.  
 Then we will propose a six-valued semantics for \letk, based on Nmatrices,  
 with the semantic values $T, T_0, \bo, \nei, F_0, \mbox{ and } F$,  which 
  correspond to the six scenarios expressed by the \lets\   just seen above.  
 We will see that the six-valued semantics combines these six scenarios 
in a sensible way but,  as expected, it  is non-deterministic.  

From here on, consider a denumerable set $\mathcal{V}$ of propositional variables, and let $For(\Theta)$ be the algebra of formulas freely generated by $\mathcal{V}$ over a propositional signature $\Theta$.

 \begin{definition} \label{def.ND.letk} 
Consider the propositional signature  $\Sigma=\{\land,\lor,\to,\neg,\cons\}$.
A natural deduction system 
over $\Sigma$ for the logic \letk\ is given by the following inference rules:

\mm

\begin{center}
 $\infer [{\land I}] {A \land B} {A & B}$ \hspace{1.6cm}
$\infer [{\land E}]  {A} {A \land B}  \hspace{2 mm} \infer []  {B} {A \land B}  $ \hspace{6mm}

\

$\infer [{\lor I}] {A \lor B} {A} \hspace{2 mm} \infer [] {A \lor B} {B}$ \hspace{1cm}
$\infer [{\lor E}] {C} {A \lor B & \infer*{C} {[A]} & \infer*{C} {[B]}}$\hspace{6mm}

\

$\infer [{\neg {\land} I}] {\neg(A \land B)} {\neg A } \hspace{2 mm} \infer [] {\neg(A \land B)} {\neg B}$\hspace{1cm}
$\infer [{\neg {\land} E}] {C} {\neg(A \land B) & \infer*{C} {[\neg A]} & \infer*{C} {[\neg B]}}$\hspace{6mm}

\

$\infer [{\neg {\lor} I}] {\neg(A \lor B)} {\neg A & \neg B}$ \hspace{1.2cm}
$\infer [{\neg {\lor} E}]  {\neg A} {\neg (A \lor B)}  \hspace{2 mm} \infer []  {\neg B} {\neg (A \lor B)}  $ \hspace{1mm}

\ 

$\infer [{DN}] {\neg \neg A} {A}  \hspace{2 mm} \infer [] {A} {\neg \neg A} $

\

		$\infer [{{\rightarrow} I}] {A \rightarrow B} {\infer*{B} {[A]}}$ \hspace{1cm}
		$\infer [\rightarrow E] {B} {A \rightarrow B & A}$ \hspace{7mm}
		$\infer [{\to_{CL}}] {A \lor (A \to B)} {} $

		\

		\m

		$\infer [{\neg {\rightarrow} I}] {\neg(A \rightarrow B)} {A & \neg B}$\hspace{1.5cm}
		$\infer [{\neg {\rightarrow E}}] {A} {\neg (A \rightarrow B)} \hspace{2 mm} \infer [] {\neg B} {\neg (A \rightarrow B)}$

\

$\infer [{EXP^{\circ}}] {B} {\circ A & A & \neg A }$ \hspace{10 mm} 
 $\infer [{PEM^\con}] {A \vee \neg A} {\con A}$\hspace{6mm}

\end{center}

\m\mh A deduction of  $A$ from a set of premises $\Gamma$ is defined as usual  for natural deduction systems. We write $\Gamma\vdash_{\letk} A$ to denote that there is one of such deductions.

\m \mh
The rules $\lor E$, $\neg {\land} E$,  and ${\rightarrow} I$ are called {\em improper}, while the others are {\em proper}.  The axiom $\to_{CL}$ is neither a proper nor an improper inference rule. 

\end{definition}

\subsection{Valuation semantics for \letk} \label{sec.bivalued.letk}

A sound and complete bivalued non-deterministic  semantics (a valuation semantics) for \letk\ is given below (cf. \cite[Sects.~3.3.1-2]{kolkata}):

\begin{definition}  \label{def-val-LETK} 
A  {\em bivaluation} for \letk\ is a function $\rho:For(\Sigma) \to \{0,1\}$ satisfying the following properties:\\[1mm]
$\begin{array}{ll}
\mbox{(v1)} & \mbox{$\rho(A \land B)=1$ iff $\rho(A)=1$ and $\rho(B)=1$};\\[1mm]
\mbox{(v2)} & \mbox{$\rho(A \lor B)=1$ iff $\rho(A)=1$ or $\rho(B)=1$};\\[1mm]
\mbox{(v3)} &  \mbox{$\rho(A \to B)=1$ iff $\rho(A)=0$ or $\rho(B)=1$};\\[1mm]
\mbox{(v4)} & \mbox{$\rho(\neg\neg A) = 1$ iff $\rho(A) = 1$};\\[1mm]
\mbox{(v5)} & \mbox{$\rho(\neg (A \land B)) = 1$ iff $\rho(\neg A) = 1$ or  $\rho(\neg B) = 1$};\\[1mm]
\mbox{(v6)} & \mbox{$\rho(\neg (A \lor B)) = 1$ iff $\rho(\neg A) = 1$ and  $\rho(\neg B) = 1$};\\[1mm]
\mbox{(v7)} & \mbox{$\rho(\neg (A \to B)) = 1$ iff $\rho(A) = 1$ and  $\rho(\neg B) = 1$};\\[1mm]
\mbox{(v8)} & \mbox{if $\rho(\cons A)=1$, then: $\rho(\neg A)=1$ iff $\rho(A)=0$}.
\end{array}$\\[1mm]

\m 

\noi The semantical consequence relation $\models_{\letk}^2$ of \letk\ with respect to  bivaluations is defined as follows:  $\Gamma\models_{\letk}^2 A$ \ if and only if, for every bivaluation $\rho$ for \letk, if $\rho(B)=1$ for every $B \in \Gamma$ then $\rho(A)=1$.\footnote{The clauses (v1), (v2), (v4), (v5), and (v6)  of Definition~\ref{def-val-LETK} constitute a sound and complete valuation semantics for \fde. See \cite[Sect.~2.1]{letf}.}
\end{definition}

\begin{remark} \label{val-Boole} \ 
 
\m\mh (1) The bivalued semantics for \letk\ makes it clear that   
the semantic values of $\neg A$ and $\con A$ are not  functionally determined by the value of $A$. 
This is in line with the idea of  the six scenarios expressed by \lets\ 
presented in Section~\ref{sec.lets.fde}, where  $A$ holds is to be read $\rho(A)=1$, and the fact that \letk\ is not (deterministically) finite-valued. In Section~\ref{sec.six.valued.semantics.LETK.decid} a six-valued  non-deterministic semantics for \letk\ will be considered, in which the truth-values are triples formed by the values of $A$, $\neg A$, and $\con A$ in a given bivaluation.

\m\mh (2) The 2-element Boolean algebra with domain ${\bf 2}=\{0,1\}$ will be denoted by  $\mathcal{B}_2$, and its operations will be denoted by $\sneg$ (Boolean complement),  $\sqcap$ (infimum), and $\sqcup$ (supremum). The implication will be defined as usual as $a \Rightarrow b =\sneg a \sqcup b$. It is well known that $\mathcal{B}_2$ is the {\em generator} of the variety of Boolean algebras, which implies that a given equation in the signature of Boolean algebras holds in $\mathcal{B}_2$ iff it holds in {\em every} Boolean algebra (this fact will be used later in the proof of Theorem~\ref{sound-compLETKPswap}). 
\end{remark}

\begin{proposition} A function $\rho:For(\Sigma) \to {\bf 2}$ is a bivaluation for \letk\ if and only if it satisfies the following properties, expressed in the language of Boolean algebras:\\[1mm]


$\begin{array}{ll}
\mbox{(v1)}' & \rho(A \land B)= \rho(A) \sqcap \rho(B);\\[1mm]
\mbox{(v2)}' &  \rho(A \lor B)= \rho(A) \sqcup \rho(B);\\[1mm]
\mbox{(v3)}' &   \rho(A \to B)= \rho(A) \Rightarrow \rho(B)=\sneg \rho(A) \sqcup \rho(B);\\[1mm]
\mbox{(v4)}' & \rho(\neg\neg A) = \rho(A);\\[1mm]
\mbox{(v5)}' & \rho(\neg (A \land B)) = \rho(\neg A) \sqcup \rho(\neg B);\\[1mm]
\mbox{(v6)}' & \rho(\neg (A \lor B)) = \rho(\neg A) \sqcap \rho(\neg B);\\[1mm]
\mbox{(v7)}' & \rho(\neg (A \to B)) = \rho(A) \sqcap \rho(\neg B);\\[1mm]
\mbox{(v8)}' & \rho(\cons A) \leq \rho(A) \sqcup \rho(\neg A) \ \mbox{ and } \ \rho(A) \sqcap \rho(\neg A) \sqcap \rho(\cons A)=0.
\end{array}$
\end{proposition} 

\noi It is straightforward to  see that clauses (v1)$'-$(v7)$'$ correspond to clauses (v1)-(v7) of Definition~\ref{def-val-LETK}. 
Concerning the clause (v8)$'$, note that it corresponds  to the rules $PEM^\con$ and $EXP^\con$. 

\subsubsection{Soundness and completeness}

The proofs of soundness and completeness of \letk\ with respect to the semantics above can be found in 
 \citep{kolkata}. However, in order to keep this paper as self-contained as possible, we will provide the main ideas behind these proofs, which will be adapted and extended to the system \letkp\ to be studied in the following  sections.

\begin{definition} \label{def-F-sat-LETK}
Let $\Delta$ be a set of formulas over $\Sigma$, and let $F$ be a formula over $\Sigma$. The set $\Delta$ is said to be {\em $F$-saturated in \letk} if: (1) $\Delta \nvdash_{\letk} F$; and (2) if $A \notin \Delta$ then $\Delta,A \vdash_{\letk} F$. 
\end{definition}

\begin{remark} \label{Lindenbaum-Los}
From a very general result for logic systems due to Lindenbaum and \L o\'s (see, for example, \cite[Theorem 22.2]{woj:84} or~\cite[Theorem~2.2.6]{car.con.2016.book}), 
the following property holds in \letk\ (in fact, it holds in any Tarskian and finitary logic):\footnote{Recall that a logic {\bf L} defined over a language $\mathcal{L}$ and with a consequence relation $\vdash$ is said to be  {\em Tarskian} if it satisfies the following properties, for every set of formulas $\Gamma \cup \Delta \cup \{A\}\subseteq \mathcal{L}$:~(i)~if $A \in \Gamma$ then $\Gamma \vdash A$; (ii) if $\Gamma \vdash A$ and $\Gamma \subseteq \Delta$ then $\Delta \vdash A$; and (iii)~if $\Delta \vdash A$ and $\Gamma \vdash B$ for every $B \in \Delta$ then $\Gamma \vdash A$.	The logic {\bf L} is said to be {\em finitary} if the following holds: if $\Gamma \vdash A$ then there exists  a finite subset $\Gamma_0$ of $\Gamma$ such that $\Gamma_0 \vdash A$.} 

\begin{itemize}
\item[] If $\Gamma \nvdash_{\letk} A$, then there exists a set $\Delta$ such that $\Gamma \subseteq \Delta$ and $\Delta$ is $A$-saturated in \letk.
\end{itemize}
\end{remark}

\begin{proposition}\label{F-sat-LETK} 
Let $\Delta$ be an $F$-saturated set in \letk. Then:\\[1mm]
(1) $A \in \Delta$ iff $\Delta \vdash_{\letk} A$; \\[1mm]
(2) $A \land B \in \Delta$ iff $A \in \Delta$ and $B \in \Delta$; \\[1mm]
(3) $A \lor B \in \Delta$ iff $A \in \Delta$ or $B \in \Delta$;\\[1mm]
(4) $A \to B \in \Delta$ iff $A \notin \Delta$ or $B \in \Delta$;\\[1mm]
(5) $\neg(A \land B) \in \Delta$ iff $\neg A \in \Delta$ or $\neg B \in \Delta$;\\[1mm]
(6) $\neg(A \lor B) \in \Delta$ iff $\neg A \in \Delta$ and $\neg B \in \Delta$;\\[1mm]
(7) $\neg(A \to B) \in \Delta$ iff $A \in \Delta$ and $\neg B \in \Delta$;\\[1mm]
(8) $\neg\neg A \in \Delta$  iff $A \in \Delta$;\\[1mm]
(9) If $\cons A \in \Delta$  then: either $A \in \Delta$ or $\neg A \in \Delta$, but not both.
\begin{proof} Left to the reader (see  \citep{kolkata}). 
\end{proof}\end{proposition}

\begin{corollary} \label{val-F-sat-LETK}
Let $\Delta$ be a set of formulas which is  $F$-saturated in \letk. Let $\rho_\Delta:For(\Sigma) \to {\bf 2}$ be the characteristic function of $\Delta$, that is: for every formula $A$, $\rho_\Delta(A)=1$ iff $A \in \Delta$ (iff $\Delta \vdash_{\letk} A$, by Proposition~\ref{F-sat-LETK}(1)). Then, $\rho_\Delta$ is a bivaluation for \letk.
\begin{proof}
It is an immediate consequence of Proposition~\ref{F-sat-LETK}, items~(2)-(9).
\end{proof}
\end{corollary}

\begin{theorem}  [Soundness and completeness of \letk\ w.r.t. bivaluation semantics] \label{adeq-LETK-bival}  \ 

\m\noi
For every set of formulas $\Gamma \cup \{A\} \subseteq For(\Sigma)$: 
$\Gamma \vdash_{\letk} A$ \ iff \  $\Gamma\models_{\letk}^2 A$.
\begin{proof} We will present only an sketch of the proof (see   details in  \citep{kolkata}). \ \\[1mm]
{\em `Only if' part (soundness)}: Let $\rho$ be a bivaluation for \letk. It is immediate to see that, for every instance $A$ of an axiom of \letk, $\rho(A)=1$. On the other hand, for every proper rule of \letk\ (i.e., for every rule which does not discharge hypotheses), if $\rho(A)=1$ for every premise of the rule then $\rho(B)=1$ for the consequence of the rule. Finally, if $B$ is the consequence of an improper rule (that is, a rule which depends on previous derivations in which some  hypotheses are discharged), it is also immediate to see, by induction hypothesis and by definition of $\rho$, that $\rho(B)=1$ provided that $\rho$ satisfies all the assumptions of the rule.\\[1mm]
{\em `If' part (completeness)}: Suppose that $\Gamma \nvdash_{\letk} A$.  By Remark~\ref{Lindenbaum-Los}, there exists a set $\Delta$ such that $\Gamma \subseteq \Delta$ and $\Delta$ is $A$-saturated in \letk. By Corollary~\ref{val-F-sat-LETK}, the characteristic function $\rho_\Delta$ of  $\Delta$ is a bivaluation for \letk\ such that $\rho_\Delta(B)=1$ for every $B \in \Gamma$, but $\rho_\Delta(A)=0$. This shows that $\Gamma \not\models_{\letk}^2 A$.
\end{proof}\end{theorem}

\subsection{A six-valued (non-deterministic) semantics for \letk} \label{sec.six.valued.semantics.LETK.decid}

  In this section we will present a six-valued non-deterministic semantics for \letk\  based on \textit{Nmatrices}. 
 The latter  are defined  from the above bivalued semantics for \letk\  (Definition \ref{def-val-LETK}) 
 by means of {\em swap structures}, as described below.

 A non-deterministic matrix (Nmatrix, for short) is a \textit{multialgebra}  together with a non-empty subset of its domain, which is the set  of {\em designated} values. A multialgebra is  an algebraic structure  equipped with at least one \textit{multioperation}. The latter is  
  an operation that returns, for  each input, a non-empty set of values instead of a single value.   
 Nmatrices are a generalization of logical matrices in which each entry of the matrices interpreting the connectives can produce a non-empty set of possible semantic values. 
 The valuations defined by means of Nmatrices choose an arbitrary value on such sets, returning thus a single semantic value for each formula. Nmatrices were formally introduced by Avron and Lev  \cite{avr:lev:01,avr:lev:05}, but there are several earlier examples in the literature of the use of non-deterministic matrices 
    (see e.g.~\cite{ivl:73,ivl:88,kear:81,res:62}).

 A swap structure  for \letk\ 
 is a multialgebra  whose domain is formed by triples $(z_1,z_2,z_3)$, called {\em snapshots}, over a    
  Boolean algebra. 
 Each snapshot 
 represents a three-dimensional  semantic value in which the first coordinate $z_1$ represents the 
 (one-dimensional) semantic value of a formula $A$ in a given bivaluation $\rho$, while the other coordinates $z_2$ and $z_3$ represent  the  semantic values of $\neg A$ and $\circ A$, respectively, in this same bivaluation $\rho$.  
 A swap structure for a given logic {\bf L} yields a 
  non-de\-ter\-min\-is\-tic matrix in which the set of designated values is formed 
 by the snapshots $z$ such that $z_1=1$, which means that the formula in the position $z_1$ holds, or `is true'. 
   Let us recall the formal notion of Nmatrices, as introduced in Avron \cite{avr:lev:01,avr:lev:05}:

\begin{definition}
Let $\Theta$ be a propositional signature. A non-deterministic matrix (Nmatrix, in short) is a structure $\mathcal{M}= \langle M, \textrm{D}, \mathcal{O}\rangle$ such that $M$ and \textrm{D} are non-empty sets (of truth-values and designated truth-values, respectively) such that $\textrm{D} \subseteq M$.  $\mathcal{O}$ is a map which assigns, to each $n$-ary connective $\#$ of $\Theta$, a function $\mathcal{O}(\#):M^n \to \wp(M)\setminus\{\emptyset\}$. A (legal) valuation over $\mathcal{M}$ is a function $v:For(\Theta) \to M$ such that $v(\#(A_1,\ldots,A_n)) \in  \mathcal{O}(\#)(v(A_1),\ldots,v(A_n))$ for every formula $\#(A_1,\ldots,A_n)$ over $\Theta$, where $\#$ is a $n$-ary connective. The consequence relation $\models_\mathcal{M}$ induced by $\mathcal{M}$ is defined as follows: given a set of formulas $\Gamma \cup \{A\}$,  $\Gamma \models_\mathcal{M} A$ if and only, for any valuation $v$  over $\mathcal{M}$,  $v(A) \in \textrm{D}$ provided that $v(B) \in \textrm{D}$ for every $B \in \Gamma$.
\end{definition}

\begin{remark} Besides Nmatrices, another interesting non-deterministic semantical framework for non-classical logics is the so-called {\em possible-translations semantics} (PTS), introduced by Carnielli in~\cite{car:90,car:00}. In~\cite{car:con:05}, Theorem~35, it was shown that  Nmatrix semantics can be described by means of a suitable PTS. Conversely, in Theorem~37 it was shown that PTS for a structural logic can be described by Nmatrix semantics. Of course this does not mean that they are similar as {\em decision procedures}: for instance, as shown in~\cite{avr:07}, da Costa's logic $C_1$ cannot be characterized by a single finite Nmatrix, while admitting a finite PTS that provides a decision procedure for it (see~\cite{car.con.2016.book}, Chapter~6). It is worth noting that $C_1$ also admits a characterization by a finite-valued restricted swap structures semantics, i.e., a finite Nmatrix generated by a swap structure where the valuations are restricted, thus producing  an alternative decision procedure for $C_1$ and for every $C_n$, see~\cite{con:tol:2022}.
\end{remark}

Now, from the bivalued semantics of \letk\ (Definition \ref{def-val-LETK}),  
a six-valued swap structure can be defined in a  natural way, yielding  a six-valued Nmatrix for \letk.
 Recall from Remark~\ref{val-Boole}(2) the 
 two-element Boolean algebra  $\mathcal{B}_2= \langle {\bf 2}, \sqcap,\sqcup,\Rightarrow,\sneg,0,1 \rangle$. We denote by ${\bf 2}^3$ the set of triples $z=(z_1,z_2,z_3)$ over ${\bf 2}$.

\begin{definition} \label{defNmatLETK}
The Nmatrix  $\mathcal{M}_{\letk}= \langle \textsc{B}_{\letk}, \textrm{D}, \mathcal{O}\rangle$  for \letk\ over the Boolean algebra $\mathcal{B}_2$ is defined as follows: its domain is the set 
$$\textsc{B}_{\letk}=\{z \in {\bf 2}^3 \ : \ z_3 \leq z_1 \sqcup z_2 \ \mbox{ and } \ z_1 \sqcap z_2 \sqcap z_3=0  \}.$$ 
That is, $\textsc{B}_{\letk} = \big\{T, \, T_0, \, \bo, \, \nei, \, F_0, \, F\big\}$, where
$$\begin{array}{cc}
T=(1,0,1) & \nei=(0,0,0)\\
T_{0}=(1,0,0) & F_{0}=(0,1,0)\\
\bo=(1,1,0) &  F=(0,1,1)
\end{array}$$
The set of designated elements of $\mathcal{M}_{\letk}$ is $\textrm{D}=\big\{z \in \textsc{B}_{\letk} \ : \ z_1=1\big\} = \big\{T, \, T_0, \, \bo\big\}$, while $\textrm{ND}=\big\{z \in \textsc{B}_{\letk} \ : \ z_1\neq 1\big\} = \big\{F, \, F_0, \, \nei\big\}$ is the set of non-designated truth-values. For  $\# \in \{\land, \lor, \to, \neg, \circ\}$
the multioperations $\mathcal{O}(\#)=\tilde{\#}$  are defined as follows, for every $z$ and $w$ in $\textsc{B}_{\letk}$:\footnote{Note that $\textsc{B}_{\letk}$, as well as its multioperations, can be defined {\em mutatis mutandis}  over an arbitrary Boolean algebra besides $\mathcal{B}_2$. This will be done in Section~\ref{sect-lattice}, Definition~\ref{twist-BA}.}

\begin{itemize} 
	\item[(i)] $z\,\tilde{\land}\,w = \{u\in \textsc{B}_{\letk} \ : \ u_1=z_1\sqcap w_1 \ \mbox{ and } u_2=z_2\sqcup w_2\}$;
	\item[(ii)] $z\,\tilde{\lor}\,w = \{u\in \textsc{B}_{\letk} \ : \ u_1=z_1\sqcup w_1 \ \mbox{ and } u_2=z_2\sqcap w_2\}$;
		\item[(iii)] $z\,\tilde{\to}\,w = \{u\in \textsc{B}_{\letk} \ : \ u_1=z_1\Rightarrow w_1 \ \mbox{ and } u_2=z_1\sqcap w_2\}$;
\item[(iv)] $\tilde{\neg}\,z = \{u\in \textsc{B}_{\letk} \ : \ u_1=z_2 \ \mbox{ and } u_2=z_1\}$;
	\item[(v)] $\tilde{\circ}\,z = \{u\in \textsc{B}_{\letk} \ : \ u_1=z_3\}$.
\end{itemize}
\end{definition}

\begin{remark} \label{rem.swap-LETK} \ 
 
 \m\mh (i) The domain $\textsc{B}_{\letk}$ above does not contain the triples $(0,0,1)$ and $(1,1,1)$.  
This  is justified as follows.   
A snapshot $z$ in $\textsc{B}_{\letk}$ represents a triple   $(\rho(A),\rho(\neg A),\rho(\circ A))$ for some formula $A$ 
and bivaluation $\rho$ for \letk. 
The restrictions  $z_3 \leq z_1 \sqcup z_2$ and $z_1 \sqcap z_2 \sqcap z_3=0$ comply with the  rules 
$PEM^\con$ and $EXP^\con$ and the clause (v8) of Definition~\ref{def-val-LETK}, which 
do not allow bivaluations $\rho$ such that $\rho(A) = \rho(\neg A)=0, \rho(\circ A) =1$, 
  or $\rho(A) = \rho(\neg A)=\rho(\circ A) =   1$. 
  
\m\mh (ii) As anticipated in Remark~\ref{val-Boole}(1), the snapshots in $\textsc{B}_{\letk}$ are able to express simultaneously the value of $A$, $\neg A$, and $\con A$ in a given bivaluation. That is, they are able to express (and to combine, as we will see in the item (iii) below) the six scenarios described by \lets\  (recall Section~\ref{sec.lets.fde}) as being themselves  semantic 
 values: 
$T$ and $F$ represent, respectively, reliable information $A$ and $\neg A$;  
$T_0$ and $F_0$ represent  information $A$ and $\neg A$, respectively, but not marked as reliable;  
and  $\bo$ and $\nei$ represent the scenarios with contradictory information and no information at all about $A$ or $\neg A$  (note that the values  $T$ and $F$ of \fde\ have become here $T_0$ and $F_0$).

\m\mh (iii) The multioperations of Definition \ref{defNmatLETK} can also be presented 
as follows:
\begin{itemize} 
	\item[(1)] $(z_1,z_2,z_3)\,\tilde{\land}\,(w_1,w_2,w_3) = (z_1\sqcap w_1,z_2\sqcup w_2,\_ )$;
	\item[(2)] $(z_1,z_2,z_3)\,\tilde{\lor}\,(w_1,w_2,w_3) = (z_1\sqcup w_1,z_2\sqcap w_2,\_ )$;
	\item[(3)] $(z_1,z_2,z_3)\,\tilde{\to}\,(w_1,w_2,w_3) = (z_1\Rightarrow w_1,z_1\sqcap w_2,\_ )$;
	\item[(4)] $\tilde{\neg}\,(z_1,z_2,z_3) = (z_2,z_1,\_ )$;
	\item[(5)] $\tilde{\circ}\,(z_1,z_2,z_3) = (z_3,\_ ,\_ )$.
\end{itemize}

\noi Here, `$\_$' means that the respective coordinate can be filled arbitrarily with 0 or 1, 
provided that the resulting triple belongs to $\textsc{B}_{\letk}$,   i.e., the triple  is a snapshot for \letk. 
 To illustrate  how these multioperations combine the six scenarios  expressed by  \lets, consider a sentence $A \land B$ in a scenario 
such that there is   positive information  $A$ marked as reliable and negative information $B$ not marked as 
reliable.  
This would be expressed by a bivaluation $\rho$ such that:  
\begin{itemize}
\item[] $\rho(A)= 1, \rho(\neg A)= 0$, $\rho(\con A)=1 \mbox{ and }   \rho(B)=0$, $\rho(\neg B)=1, \rho(\con B)= 0,  $
\end{itemize}
\noi and the corresponding snapshots 
\begin{itemize}
\item[] $ ( 1,  0 ,  1) = T \mbox{ and }   (0,1,0) = F_0.$
\end{itemize}
\noi From Definition~\ref{def-val-LETK} (bivaluations),  we have that 
\begin{itemize}
\item[] $ \rho(A \land B) = 0$,  $\rho(\neg(A \land B)) = 1$,  and  $ \rho(\con (A \land B))$ is arbitrary, 
\end{itemize}
\noi 
which is expressed by the triple $(0,1,\_)$, and the latter 
 is exactly what is given by the operation 
 $\tilde{\land}$ applied to the triples $(1,0,1)$ and  $(0,1,0)$. 
 In terms of the six semantic values, we have that when $v(A)= T$ and $v(B) = F_0$, $v(A\land B)$ is either $F$ 
 or $F_0$. 
\end{remark}

\subsubsection{Non-deterministic tables for \letk}  \label{tablesLETKP}

The multioperations (1)-(5) above    can also be  described by means of the following non-deterministic 
tables:

\

{\tiny \begin{center}
\noindent
\begin{tabular}{|c|c|c|c|c|c|c|} 
\hline
 $\tilde{\wedge}$ & $T$  & $T_0$   & $\bo$ & $\nei$ & $F_0$  & $F$\\[1mm]
 \hline \hline
    $T$    & $T, T_0$  & $T, T_0$ & $\bo$ & $\nei$ & $F, F_0$ & $F, F_0$   \\[1mm] \hline
     $T_0$    & $T, T_0$  & $T, T_0$ & $\bo$ & $\nei$ & $F, F_0$ & $F, F_0$  \\[1mm] \hline
     $\bo$    & $\bo$  & $\bo$ & $\bo$ & $F, F_0$ & $F, F_0$ & $F, F_0$  \\[1mm] \hline
     $\nei$    & $\nei$  & $\nei$ & $F, F_0$ & $\nei$ & $F, F_0$ & $F, F_0$  \\[1mm] \hline
     $F_0$    & $F, F_0$  & $F, F_0$ & $F, F_0$ & $F, F_0$ & $F, F_0$ & $F, F_0$  \\[1mm] \hline
     $F$    & $F, F_0$  & $F, F_0$ & $F, F_0$ & $F, F_0$ & $F, F_0$ & $F, F_0$  \\[1mm] \hline
\end{tabular}
\end{center}
 }

\

\

{\tiny
\begin{center}
\begin{tabular}{|c|c|c|c|c|c|c|}
\hline
 $\tilde{\vee}$ & $T$  & $T_0$  & $\bo$ & $\nei$ & $F_0$  & $F$ \\[1mm]
 \hline \hline
    $T$    & $T, T_0$  & $T, T_0$ & $T, T_0$ & $T, T_0$ & $T, T_0$ & $T, T_0$   \\[1mm] \hline
     $T_0$    & $T, T_0$  & $T, T_0$ & $T, T_0$ & $T, T_0$ & $T, T_0$ & $T, T_0$   \\[1mm] \hline
     $\bo$    & $T, T_0$  & $T, T_0$ & $\bo$ & $T, T_0$ & $\bo$ & $\bo$   \\[1mm] \hline
     $\nei$    & $T, T_0$  & $T, T_0$ & $T, T_0$ & $\nei$ & $\nei$ & $\nei$   \\[1mm] \hline
     $F_0$    & $T, T_0$  & $T, T_0$ & $\bo$ & $\nei$ & $F, F_0$ & $F, F_0$   \\[1mm] \hline
     $F$    & $T, T_0$  & $T, T_0$ & $\bo$ & $\nei$ & $F, F_0$ & $F, F_0$  \\[1mm] \hline
\end{tabular}
\end{center}
 }

\

\

{\tiny
\begin{center} \label{tables.letk}
\begin{tabular}{|c|c|c|c|c|c|c|}
\hline
 $\tilde{\to}$ & $T$  & $T_0$  & $\bo$ & $\nei$ & $F_0$  & $F$ \\[1mm]
 \hline \hline
    $T$    & $T, T_0$  & $T, T_0$ & $\bo$ & $\nei$ & $F, F_0$ & $F, F_0$   \\[1mm] \hline
     $T_0$    & $T, T_0$  & $T, T_0$ & $\bo$ & $\nei$ & $F, F_0$ & $F, F_0$  \\[1mm] \hline
     $\bo$    & $T, T_0$  & $T, T_0$ & $\bo$ & $\nei$ & $F, F_0$ & $F, F_0$  \\[1mm] \hline
     $\nei$    & $T, T_0$  & $T, T_0$ & $T, T_0$ & $T, T_0$ & $T, T_0$ & $T, T_0$  \\[1mm] \hline
     $F_0$    & $T, T_0$  & $T, T_0$ & $T, T_0$ & $T, T_0$ & $T, T_0$ & $T, T_0$  \\[1mm] \hline
     $F$    & $T, T_0$  & $T, T_0$ & $T, T_0$ & $T, T_0$ & $T, T_0$ & $T, T_0$  \\[1mm] \hline
\end{tabular}
\hspace{0.3cm}
\begin{tabular}{|c||c|} \hline
$\quad$ & $\tilde{\neg}$ \\[1mm]
 \hline \hline
    $T$   & $F, F_0$    \\[1mm] \hline
     $T_0$   & $F, F_0$    \\[1mm] \hline
     $\bo$   &$\bo$    \\[1mm] \hline
     $\nei$   & $\nei$    \\[1mm] \hline
     $F_0$   & $T, T_0$    \\[1mm] \hline
     $F$   & $T, T_0$    \\[1mm] \hline
\end{tabular}
\hspace{0.3cm}
\begin{tabular}{|c||c|}
\hline
 $\quad$ & $\tilde{\circ}$ \\[1mm]
 \hline \hline
    $T$   & \textrm{D}    \\[1mm] \hline
     $T_0$   & \textrm{ND}    \\[1mm] \hline
     $\bo$   & \textrm{ND}    \\[1mm] \hline
     $\nei$   & \textrm{ND}    \\[1mm] \hline
     $F_0$   & \textrm{ND}    \\[1mm] \hline
     $F$   & \textrm{D}    \\[1mm] \hline
\end{tabular}
\end{center}
}

\

\mh For simplicity, we write $X$ and $X,Y$ instead of $\{X\}$ and $\{X,Y\}$. 
Recall that $\textrm{D}$ and $\textrm{ND}$ denote, respectively, the set of  
designated and  non-designated values.

\subsubsection{Soundness and completeness of the six-valued semantics}

The six-valued  semantics above,  as expected, is equivalent to the bivalued semantics 
of Definition   \ref{def-val-LETK}. 
Recall that, if $z \in \textsc{B}_{\letk}$ then $z_i$ denotes the $ith$-coordinate of the triple $z$. Then:

\begin{proposition} \label{lemma-sound-LETK}
For every valuation $v$  over the Nmatrix $\mathcal{M}_{\letk}$  the mapping $\rho_v:For(\Sigma) \to {\bf2}$ given by $\rho_v(A)=v(A)_1$ is a bivaluation for \letk\ such that: $\rho_v(A)=1$ \ iff \ $v(A) \in \textrm{D}$, for every formula $A$.
\begin{proof} Let $A,B \in For(\Sigma)$. Since $v$ is a valuation over $\mathcal{M}_{\letk}$ it satisfies the following:
$$\begin{array}{ll}
	v(\#\,A) \in  \tilde{\#}\,v(A) & \mbox{ for $\# \in \{\neg,\cons\}$}\\
	v(A \,\#\, B) \in v(A) \,\tilde{\#}\, v(B) & \mbox{ for $\# \in \{\land,\lor,\to\}$}\\
	\end{array}$$
Thus, by definition of $\rho_v$, by Definition~\ref{defNmatLETK} and by the fact that $v$ is a valuation over  $\mathcal{M}_{\letk}$, $\rho_v(\neg A)=v(\neg A)_1=v(A)_2$ and $\rho_v(\cons A)=v(\cons A)_1=v(A)_3$. Moreover: 
	$$\begin{array}{lll}
	\rho_v(A \land B) &=&  v(A\land B)_1=v(A)_1 \sqcap v(B)_1=\rho_v(A) \sqcap \rho_v(B)\\
	\rho_v(A \lor B) &=&  v(A\lor B)_1=v(A)_1 \sqcup v(B)_1=\rho_v(A) \sqcup \rho_v(B)\\
	\rho_v(A \to B) &=&  v(A\to B)_1=v(A)_1 \Rightarrow v(B)_1=\rho_v(A) \Rightarrow \rho_v(B)
	\end{array}$$
Hence, by Remark~\ref{val-Boole}(2), $\rho_v$ satisfies clauses (v1)-(v3). On the other hand, $\rho_v(\neg\neg A)=v(\neg A)_2=v(A)_1=\rho_v(A)$ and so $\rho_v$ satisfies (v4). In addition, 
$$\rho_v(\neg(A \land B))=v(\neg(A \land B))_1=v(A \land B)_2=v(A)_2 \sqcup v(B)_2=\rho_v(\neg A) \sqcup \rho_v(\neg B)$$
and so $\rho_v$ satisfies clause (v5), by Remark~\ref{val-Boole}(2). Analogously, it is proven that $\rho_v$ satisfies clauses (v6) and (v7). Finally, since $v(A) \in \textsc{B}_{\letk}$ then $v(A)_3 \leq v(A)_1 \sqcup v(A)_2$ and $v(A)_1 \sqcap v(A)_2 \sqcap v(A)_3=0$. That is, $\rho_v(\cons A) \leq \rho_v(A) \sqcup \rho_v(\neg A)$ and $\rho_v(A) \sqcap \rho_v(\neg A) \sqcap \rho_v(\cons A)=0$. By Remark~\ref{val-Boole}(2), this means that $\rho_v$ satisfies (v8). This shows that $\rho_v$ is a bivaluation for \letk\ such that, by definition,  $\rho_v(A)=1$ \ iff \ $v(A) \in \textrm{D}$, for every formula $A$.
\end{proof}
\end{proposition}

\begin{theorem} [Soundness of \letk\ w.r.t. the Nmatrix $\mathcal{M}_{\letk}$] \label{sound-Nmat-LETK} \ \\
	For every set of formulas $\Gamma \cup \{A\} \subseteq For(\Sigma)$: 
	$\Gamma \vdash_{\letk} A$ \ implies that \  $\Gamma\models_{\mathcal{M}_{\letk}} A$.
\begin{proof} 
Suppose that $\Gamma \vdash_{\letk} A$. By Theorem~\ref{adeq-LETK-bival},  $\Gamma\models_{\letk}^2 A$. Let $v$ be a valuation  over the Nmatrix $\mathcal{M}_{\letk}$ such that $v(B)  \in \textrm{D}$ for every $B \in \Gamma$, and let $\rho_v$ be the bivaluation for \letk\ defined from $v$ as in Proposition~\ref{lemma-sound-LETK}. Since $\rho_v(B)=1$ for every $B \in \Gamma$ then $\rho_v(A) =1$, given that $\Gamma\models_{\letk}^2 A$. From this $v(A)\in \textrm{D}$, showing that $\Gamma\models_{\mathcal{M}_{\letk}} A$.
\end{proof}\end{theorem}

\m

The completeness of \letk\ w.r.t.~the six-valued Nmatrix semantics will be proved in a similar way, based once again on Theorem~\ref{adeq-LETK-bival}.

\begin{proposition} \label{lemma-comple-LETK}
For every bivaluation $\rho$ for \letk\  the mapping $v_{\rho}:For(\Sigma) \to \textsc{B}_{\letk}$ given by $v_{\rho}(A)=(\rho(A),\rho(\neg A),\rho(\circ A))$ is a valuation over the Nmatrix $\mathcal{M}_{\letk}$ such that: $v_{\rho}(A) \in \textrm{D}$ \ iff \ $\rho(A)=1$, for every formula $A$.
\begin{proof} Clearly $v_{\rho}(A) \in \textsc{B}_{\letk}$, hence the function is well-defined. Let $A,B \in For(\Sigma)$. Then, if $a=\rho(\circ (A \land B))$ we have: 
	$$\begin{array}{lll}
	v_{\rho}(A \land B) &=&  (\rho(A\land B),\rho(\neg(A\land B)),a)\\
	&=& (\rho(A)\sqcap \rho(B),\rho(\neg A)\sqcup \rho(\neg B),a) \in v_{\rho}(A)\,\tilde{\land}\, v_{\rho}(B).
	\end{array}$$
The cases of $\lor$ and $\to$ are proved analogously. Concerning negation, let $b=\rho(\circ \neg A)$. Then:
	$$\begin{array}{lll}
	v_{\rho}(\neg A) &=&  (\rho(\neg A),\rho(\neg\neg A),b)\\
	&=& (\rho(\neg A),\rho(A),b) \in \tilde{\neg}\, v_{\rho}(A).
	\end{array}$$
Finally, if $a=\rho(\neg\cons A)$ and $b=\rho(\cons\cons A)$ then
	$$\begin{array}{lll}
	v_{\rho}(\cons A) &=&  (\rho(\cons A),a,b)  \in \tilde{\cons}\, v_{\rho}(A).
	\end{array}$$
		This shows that $v_\rho$ is a bivaluation for \letk\ such that, by definition,   $v_\rho(A) \in \textrm{D}$ \ iff \ $\rho(A)=1$, for every formula $A$.       
\end{proof}\end{proposition}

\begin{theorem} [Completeness of \letk\ w.r.t. the Nmatrix $\mathcal{M}_{\letk}$] \label{comp-Nmat-LETK} \ \\
	For every set of formulas $\Gamma \cup \{A\} \subseteq For(\Sigma)$: 
	$\Gamma \models_{\mathcal{M}_{\letk}} A$ \ iff \  $\Gamma \vdash_{\letk} A$.
\begin{proof} Assume that $\Gamma\models_{\mathcal{M}_{\letk}} A$, and let $\rho$ be a bivaluation for \letk\ such that $\rho(B)=1$ for every $B \in \Gamma$. Let $v_{\rho}$ be defined from $\rho$ as in Proposition~\ref{lemma-comple-LETK}. Then, $v_{\rho}$ is a valuation over $\mathcal{M}_{\letk}$ such that $v_{\rho}(B) \in \textrm{D}$, for every $B \in \Gamma$. By hypothesis, $v_{\rho}(A) \in \textrm{D}$, whence $\rho(A)=1$. This shows that  $\Gamma\models_{\letk}^2 A$. By completeness of \letk\ w.r.t. bivaluations, $\Gamma \vdash_{\letk} A$.
\end{proof}\end{theorem}

\begin{corollary} 
The six-valued  Nmatrix $\mathcal{M}_{\letk}$ provides a decision procedure for \letk.
\begin{proof}
It is well known that finite Nmatrices, like the one above, provide  
 decision procedures. This is because  the number of sentential variables is finite, 
 the number of bifurcated lines  is always finite, and clearly there is no loop in the procedure. 
\end{proof}
\end{corollary}

\section{Adding propagation rules to \letk: the logic \letkp} \label{sec.letkp}

In a broad sense, propagation of  classicality is how classical   behavior propagates from less complex to more complex sentences, 
and vice-versa. The   \lets\ investigated so far (\letj, \letf, \letk, \letfm) enjoy the following property:

\begin{proposition}\label{propag.cla} Let $\textbf{L} \in \{ \mbox{\letfm, \letj, \letk, \letf} \}$ and $\Gamma = \{\circ \neg^{n_1}A_1, \dots, \circ \neg^{n_m}A_m \}$, for $n_i \geq 0$ (where $ \neg^{n_i} $,  
		$ n_i \geq 0 $, represents $ n_i $ occurrences  of negations before the formula $ A_i $). Then, for any formula $B$ formed with $A_1, \dots, A_m $ over the  signature $\{ \neg, \land, \lor, \to\}$ (and $\{ \neg, \land, \lor  \}$ in the case of \letf\ and \letfm), 
$\Gamma\vdash_{\textbf{L}} B \lor \neg B $, and $\Gamma, B , \neg B \vdash_{\textbf{L}} C$. 
That is, $B$ behaves classically in this context. 

	\begin{proof} This result is proved for \letf\ in   \cite[Fact~31]{letf}, and for \letj\ in \cite[Proposition~7]{axioms}. 
	Proofs for \letfm\ and \letk\ can be obtained similarly. 
	\end{proof} \end{proposition}

\noi However, although in the \lets\ studied so far the classical behavior is transmitted 
from less complex to more complex formulas, 
the classicality operator \con\ is not;  
that is, the  inferences 
$\cons p \vdash \cons\neg p$ and $\cons p,\cons q \vdash \cons(p \,\#\, q)$ ($\# \in \{\lor, \land \}$), 
for example,  do not hold.  
 This should be clear in \letk\    from the fact that there is no introduction rule for \con, 
 but    the reader can also check this from the six-valued matrices presented in 
 Section~\ref{sec.six.valued.semantics.LETK.decid}, 
 which explicitly display 
the non-deterministic behavior of \con. 
 But to rigorously express   the idea of dividing the sentences of the language into two groups, 
which is an essential 
point of \lets\  (as well as  \lfis\ and  da Costa's $C_n$ hierarchy)  
it would be desirable for the classicality operator \con\  to be transmitted as well.

We find in the literature two ways of establishing the propagation of \con, 
but neither fits  
 the intended interpretation  of \lets.  
 In da Costa's $C_n$ hierarchy, the well-behavior of $A$ (originally represented by $A^\con$) 
propagates as follows:  

\begin{enumerate}[resume]
\item  $\con A \land \con B \vdash \con(A \,\#\, B), \mbox{ for } \#  \in \{ \land, \lor, \to \}. $ \label{prop.newton}
\end{enumerate}

\noi Although   \ref{prop.newton} fits  with   Proposition \ref{propag.cla} above, it places  a too strong condition 
on the  propagation of classicality.  
Indeed, we will see that it is not always necessary that both $\con A$ and $\con B$ hold for concluding $\con (A \,\#\, B)$. 
On the other hand,  in the logic  \textit{Cilo}, an \lfi\ investigated in  \cite[pp.~116ff.]{car.con.2016.book}, consistency propagates as follows:    
\begin{enumerate}[resume]
\item  $\con A \lor \con B \vdash \con(A \,\#\, B), \mbox{ for } \#  \in \{ \land, \lor, \to \}. $ \label{prop.bez}
\end{enumerate}
\noi The condition for propagation in \ref{prop.bez}, however, is too weak in the sense that 
it allows one to conclude $\con (A \,\#\, B)$ in circumstances where it should not be concluded. 
Let us illustrate this by taking a look at how \con\ should be transmitted over $\lor $.

Recall that  $\con A \land A$  and   $\con A \land \neg A$  
are intended to mean  that the  information conveyed, respectively,  by $A$ 
 and by $\neg A$, is considered reliable. 
  In addition, 
  positive and negative reliable  information behaves like truth and falsity in classical logic.  
 As a consequence, from  $\con A \land A$ one should be able to infer  that $A \lor B$ is also 
reliable  for any $B$, no matter whether $\con B$ 
holds or not, and so $\con(A \lor B) \land (A \lor B)$ 
holds. 
For if this were not the case, that is, if both $(A \lor B)$ and $\neg(A \lor B)$ held, 
$\con A$ could not hold, 
since $\neg (A \lor B)$ implies $\neg A$. 
On the other hand, from  $\con A \land \neg A$, it cannot be inferred that $\neg(A \lor B)$ is 
reliable, because 
  to conclude $\neg(A \lor B)$ both $\neg A$ and $\neg B$ are required. Hence, from $\con A \land \neg A$, 
  $\con(A \lor B)$ cannot be inferred.
This    suggests the validity of the following inferences:   
\begin{enumerate}[resume]
\item $\con A, A \vdash \con (A \lor B) \land (A \lor B)$,  \label{inf.1}
\item $\con B, B \vdash \con (A \lor B) \land (A \lor B)$, 
\item $\con A, \neg A, \con  B, \neg B \vdash \con (A \lor B) \land \neg(A \lor B)$, \label{inf.last}  
\end{enumerate}

\noi and so on.   The point of these inferences is that  
positive (resp. negative) 
reliable information behaves like truth  (resp. falsity) 
in  classical logic: in order to have a $A \lor B$ false,  we need both $A$ and $B$ false, 
but $A$ true is enough to conclude 
$A\lor B$ true.\footnote{It is to be noted  that the fact that reliable information is subjected to classical logic,  
 and so  to the rules of preservation of truth,  
 does not mean, of course, that reliable information is  being \textit{identified}   with truth,  
 in the realist sense of the notion of truth that underlies the standard  interpretation 
 of classical logic. 
We are, rather, making the weaker claim that people reason with reliable information as if it were 
true, and so classical logic is appropriate to express the deductive behavior of reliable information.}  
  In the following, in order to   define the system \letkp, 
 we will extend this line of reasoning to  the connectives $\wneg$, 
  $\land$,  and $\to$.

\begin{definition}    \label{propag-rules} For any formula $A$, let  $A^T \defi \cons A \land A$ and $A^F \defi \cons A \land \neg A$.  

\m\mh A natural deduction system for \letkp\ is obtained by adding the following axiom and rules to the  system 
of \letk\ (Definition~\ref{def.ND.letk}):

\begin{center}

\mmmm

$\infer [_{[I \cons]}] {\cons\cons A}{}$ \hspace{12mm} $\infer[_{[I \neg\cons]}] {\cons \neg A}{\cons A}$  \hspace{12mm} $\infer [_{[E \neg\cons]}] {\cons A} {\cons \neg A}$

\mmmm

$\infer [_{[I\land T]}] {(A \land B)^T} {A^T & B^T}$ \hspace{12mm}
$\infer [_{[I\land F]}] {(A \land B)^F} {A^F} \hspace{2 mm} \infer [] {(A \land B)^F} {B^F}$ \hspace{10mm}

\mmmm

$\infer [_{[I\lor T]}] {(A \lor B)^T} {A^T} \hspace{2 mm} \infer [] {(A \lor B)^T} {B^T}$ \hspace{10mm}
$\infer [_{[I\lor F]}] {(A \lor B)^F} {A^F & B^F}$ \hspace{10mm}

\mmmm

$  \infer  [_{[I\to T]}] {(A \to B)^T} {A^F} \hspace{2mm}  \infer [] {(A \to B)^T}  {B^T} $ \hspace{10mm}
$ \infer [_{[I\to F]}]  {(A \to B)^F} {A  & B^F }$

\mmmm

 $\infer [_{[E\land T]}]  {A^T} {(A \land B)^T}  \hspace{2 mm} \infer []  {B^T} {(A \land B)^T}  $ \hspace{10mm}
$\infer [_{[E\land F]}] {C} {(A \land B)^F & \infer*{C} {[A^F]} & \infer*{C} {[B^F]}}$\hspace{6mm}

\mmmm

 $\infer [_{[E\lor T]}] {C} {(A \lor B)^T & \infer*{C} {[A^T]} & \infer*{C} {[B^T]}}$\hspace{10mm}
$\infer [_{[E\lor F]}]  {A^F} {(A \lor B)^F}  \hspace{2 mm} \infer []  {B^F} {(A \lor B)^F}  $ \hspace{8mm}

\mmmm

$\infer [_{[E\to T]}] {C} {(A \to B)^T & \infer*{C} {[A^F]} & \infer*{C} {[B^T]}}$
\hspace{10mm}
$\infer [_{[E\to F]}]  {A} {(A \to B)^F}  \hspace{2 mm} \infer []  {B^F} {(A \to B)^F}  $ \hspace{8mm}
\end{center}

\mm

\mh A deduction of  $A$ from a set of premises $\Gamma$, $\Gamma\vdash_{\letkp} A$, 
is defined as usual  for natural deduction systems. 

\end{definition}

\begin{proposition} \label{der-rules} 
The following  rules are derived in  \letkp:  

\begin{center}
$\infer [_{[Cons]}]  {B} {\cons A & \neg\cons A}$  \hspace{10mm} 
$\infer [_{[Comp]}] {B} {\infer*{B} {[\cons A]} & \infer*{B} {[\neg\cons A]}} $
\end{center}
\begin{proof}
The rules above are obtained in a few steps  from $EXP^\cons$, $PEM^\cons$,  
and $I\cons$. 
\end{proof}
\end{proposition}

\subsection{Valuation semantics for \letkp}
The   rules and the axiom of \letkp\ added to \letk\ induce, in a very natural way, a   bivalued 
semantics, which will be described in what follows  (as we shall see in Proposition~\ref{charact-bival-LETKP}, the definition below can be drastically simplified when expressed in the language of Boolean algebras).

\begin{definition} [Valuation semantics for \letkp] \label{val-sem-propag}  A bivaluation for \letkp\ is a bivaluation $\rho:For(\Sigma) \to {\bf2}$ for \letk\ (Definition~\ref{def-val-LETK}) satisfying, in addition, the following clauses: 

\begin{itemize}  
\item[] (vp1) $\rho(\cons \cons A)=1$;
\item[]  \mbox{(vp2)} $ \rho(\cons \neg A)=\rho(\cons A)$;
\item[] \mbox{(vp3)}  \mbox{If $\rho(\cons A)=\rho(A)=1$ and $\rho(\cons B)=\rho(B)=1$  then $\rho(\cons(A \land B))=1$};
\item[] \mbox{(vp4)}  \mbox{If $\rho(\cons A)=\rho(\neg A)=1$  then $\rho(\cons(A \land B))=1$};
\item[] \mbox{(vp5)}  \mbox{If $\rho(\cons B)=\rho(\neg B)=1$  then $\rho(\cons(A \land B))=1$};
\item[] \mbox{(vp6)}  \mbox{If $\rho(\cons(A \land B))=\rho(A)=\rho(B)=1$  then $\rho(\cons A) = \rho(\cons B)=1$};
\item[] \mbox{(vp7)}   \mbox{If $\rho(\cons(A \land B))=1$, and either $\rho(\neg A)=1$ or $\rho(\neg B)=1$,  then:} \\
  \mbox{\quad \quad \quad  \quad    either $\rho(\cons A) = \rho(\neg A)=1$ or  $\rho(\cons B)=\rho(\neg B)=1$};
\item[] \mbox{(vp8)}   \mbox{If $\rho(\cons A)=\rho(A)=1$  then $\rho(\cons(A \lor B))=1$};
\item[] \mbox{(vp9)}   \mbox{If $\rho(\cons B)=\rho(B)=1$  then $\rho(\cons(A \lor B))=1$};
\item[] \mbox{(vp10)}  \mbox{If $\rho(\cons A)=\rho(\neg A)=1$ and $\rho(\cons B)=\rho(\neg B)=1$   then $\rho(\cons(A \lor B))=1$};
 \item[] \mbox{(vp11)}  \mbox{If $\rho(\cons(A \lor B))=1$, and either $\rho(A)=1$ or $\rho(B)=1$,  then:} \\ 
 \mbox{\quad \quad \quad  \quad   either $\rho(\cons A) = \rho(A)=1$ or  $\rho(\cons B)=\rho(B)=1$};
\item[] \mbox{(vp12)}  \mbox{If $\rho(\cons(A \lor B))=1$ and $\rho(A)=\rho(B)=0$  then $\rho(\cons A) = \rho(\cons B)=1$};
\item[] \mbox{(vp13)}   \mbox{If $\rho(\cons A)=\rho(\neg A)=1$  then $\rho(\cons(A \to B))=1$};
\item[] \mbox{(vp14)}   \mbox{If $\rho(\cons B)=\rho(B)=1$  then $\rho(\cons(A \to B))=1$};
\item[] \mbox{(vp15)}   \mbox{If $\rho(A)=1$ and $\rho(\cons B)=\rho(\neg B)=1$   then $\rho(\cons(A \to B))=1$};
\item[] \mbox{(vp16)}   \mbox{If $\rho(\cons(A \to B))=1$, and either $\rho(A)=0$ or $\rho(B)=1$,  then:} \\ 
 \mbox{\quad\quad\quad\quad  either $\rho(\cons A) =\rho(\neg A)=1$ or  $\rho(\cons B)=\rho(B)=1$};
\item[] \mbox{(vp17)}  \mbox{If $\rho(\cons(A \to B))=1$ and $\rho(A)=\rho(\neg B)=1$  then $\rho(\cons B)=1$}.
\end{itemize}

\end{definition}

Now  we shall prove the soundness and completeness of \letkp\ w.r.t.~the bivalued semantics above. 
The definition of $F$-saturated sets in \letkp\ is analogous to the one for \letk\ (Definition~\ref{def-F-sat-LETK}).

\begin{proposition}\label{F-sat-LETKP} 
Let $\Delta$ be an $F$-saturated set in \letkp. Then, it is a closed theory (that is, $A \in \Delta$ iff $\Delta \vdash_{\letkp} A$) and it satisfies properties (2)-(9) of Proposition~\ref{F-sat-LETK} plus the following:  

\begin{itemize}  
\item[] (vp1$'$) $\cons\cons A \in \Delta$;
\item[] (vp2$'$) $\cons A \in \Delta$  iff $\cons\neg A \in \Delta$;
\item[] (vp3$'$)  If $\cons A \in \Delta$, $A \in \Delta$, $\cons B \in \Delta$ and $B\in \Delta$, then $\cons(A \land B)\in \Delta$;
\item[] (vp4$'$) If $\cons A\in \Delta$ and $\neg A\in \Delta$,  then $\cons(A \land B)\in \Delta$;
\item[] (vp5$'$) If $\cons B\in \Delta$ and $\neg B\in \Delta$,  then $\cons(A \land B)\in \Delta$;
\item[] (vp6$'$) If $\cons(A \land B)\in \Delta$, $A\in \Delta$ and $B\in \Delta$,  then $\cons A\in \Delta$ and $\cons B\in \Delta$;
\item[] (vp7$'$) If $\cons(A \land B)\in \Delta$, and either $\neg A\in \Delta$ or $\neg B\in \Delta$,  then: \\
\mbox{\quad\quad\quad\quad} either $\cons A\in \Delta$ and  $\neg A\in \Delta$ or  $\cons B\in \Delta$ and $\neg B\in \Delta$;
\item[] (vp8$'$) If $\cons A\in \Delta$ and $A\in \Delta$,  then $\cons(A \lor B)\in \Delta$;
\item[] (vp9$'$) If $\cons B\in \Delta$ and $B\in \Delta$,  then $\cons(A \lor B)\in \Delta$;
\item[] (vp10$'$) If $\cons A\in \Delta$, $\neg A\in \Delta$, $\cons B \in \Delta$ and $\neg B \in \Delta$,   then $\cons(A \lor B)\in \Delta$;
\item[] (vp11$'$) If $\cons(A \lor B)\in \Delta$, and either $A\in \Delta$ or $B\in \Delta$,  then: \\
\mbox{\quad\quad\quad\quad} either $\cons A\in \Delta$ and $A\in \Delta$ or  $\cons B\in \Delta$ and $B\in \Delta$;
\item[] (vp12$'$) If $\cons(A \lor B)\in \Delta$, $A\notin \Delta$ and $B \notin  \Delta$,  then $\cons A\in \Delta$ and $\cons B\in \Delta$;
\item[] (vp13$'$) If $\cons A\in \Delta$ and $\neg A\in \Delta$,  then $\cons(A \to B)\in \Delta$;
\item[] (vp14$'$) If $\cons B\in \Delta$ and $B\in \Delta$,  then $\cons(A \to B)\in \Delta$;
\item[] (vp15$'$) If $A\in \Delta$, $\cons B \in \Delta$ and $\neg B\in \Delta$,   then $\cons(A \to B)\in \Delta$;
\item[] (vp16$'$) If $\cons(A \to B) \in \Delta$, and either $A \notin \Delta$ or $B \in \Delta$,  then: \\
\mbox{\quad\quad\quad\quad} either $\cons A \in \Delta$ and  $\neg A \in \Delta$ or  $\cons B \in \Delta$ and $B \in \Delta$;
\item[] (vp17$'$) If $\cons(A \to B) \in \Delta$, $A \in \Delta$ and $\neg B \in \Delta$,  then $\cons B \in \Delta$.

\end{itemize}
\end{proposition}
\begin{proof}
It is immediate to see that an $F$-saturated set is a closed theory in any Tarskian logic.
The proof of items (2)-(9) follows from \letk.  
The  proofs of items vp$1'$ to vp$17'$ follow easily from the axiom and rules added to \letk\ in Definition~\ref{propag-rules}. Details are left to the reader.
\end{proof}

\begin{corollary} \label{val-F-sat}
Let $\Delta$ be a set of formulas which is  $F$-saturated in \letkp. Let $\rho_\Delta:For(\Sigma) \to {\bf 2}$ be the characteristic function of $\Delta$, that is: for every formula $A$, $\rho_\Delta(A)=1$ iff $A \in \Delta$ (iff $\Delta \vdash_{\letkp} A$, by Proposition~\ref{F-sat-LETKP}). Then, $\rho_\Delta$ is a bivaluation for \letkp.
\begin{proof}
It is an immediate consequence of Proposition~\ref{F-sat-LETKP}.  
\end{proof}\end{corollary}

\begin{theorem}  [Soundness and completeness of \letkp\ w.r.t. bivaluation semantics] \label{adeq-LETKP-bival}  \ \\
For every set of formulas $\Gamma \cup \{A\} \subseteq For(\Sigma)$: 
$\Gamma \vdash_{\letkp} A$ \ iff \  $\Gamma\models_{\letkp}^2 A$.
\end{theorem}
\begin{proof} \ \\
{\em `Only if' part (soundness)}: By Theorem~\ref{adeq-LETK-bival}, it suffices proving that any bivaluation for \letkp\ satisfies the axiom and rules of Definition~\ref{propag-rules}. But this is immediate, taking into account the other properties inherited from bivaluations for \letk. The details are left to the reader.\\[1mm]
{\em `If' part (completeness)}: It is also an extension of the proof of completeness of \letk\ w.r.t. bivaluations. Thus, suppose that $\Gamma \nvdash_{\letkp} A$.  As observed in Remark~\ref{Lindenbaum-Los}, being \letkp\ a Tarskian and finitary logic, there exists a set $\Delta$ such that $\Gamma \subseteq \Delta$ and $\Delta$ is $A$-saturated in \letkp. By Corollary~\ref{val-F-sat}, the characteristic function $\rho_\Delta$ of  $\Delta$ is a bivaluation for \letkp\ such that $\rho_\Delta(B)=1$ for every $B \in \Gamma$, but $\rho_\Delta(A)=0$. This shows that $\Gamma \not\models_{\letkp}^2 A$.
\end{proof}

\subsection{A six-valued semantics for \letkp}  \label{sec.six.valued.letkp}

 Let us recall the Nmatrix for \letk\ introduced in Section~\ref{sec.six.valued.semantics.LETK.decid}. 
 The snapshots $z=(z_1,z_2,z_3)$ represent triples of the form
 $(\rho(A), \rho(\neg A),\rho(\cons A))$, for a bivaluation $\rho$ for \letk\ and a formula $A$. 
 With this  in mind, and taking into account the definition of the six snapshots of $\textsc{B}_{\letk}$, 
 in what follows we will see how the clauses of Definition~\ref{val-sem-propag} 
 impose  restrictions on the multioperators  of the Nmatrix of \letk, 
 which turn out to be deterministic in \letkp.  

\begin{proposition} \label{prop.reasoning.6.val}
Clauses (vp3)-(vp7) of Definition~\ref{val-sem-propag} 
 impose the following restrictions  to the multioperator $\tilde{\land}$ of the Nmatrix of \letk: 
\begin{itemize} \setl 
\item[] (vp3): $T \,\tilde{\land}\, T=T$; 
\item[] (vp4)-(vp5): $F \,\tilde{\land}\, z=z \,\tilde{\land}\, F=F$, for every $z$; 
\item[] (vp6):  $z \,\tilde{\land}\, w=T$ implies  that $z=w=T$. 
Hence,  $T \,\tilde{\land}\, T_0=T_0 \,\tilde{\land}\, T=T_0 \,\tilde{\land}\, T_0=T_0$; 
\item[] (vp7):  $z \,\tilde{\land}\, w=F$ implies  that $z=F$ or $w=F$. 
Hence,  $\bo \,\tilde{\land}\, \nei=\nei \,\tilde{\land}\, \bo=F_0 \,\tilde{\land}\, z=z \,\tilde{\land}\, F_0=F_0$, for $z \neq F$.
\end{itemize}
\noi This shows that  $\tilde{\land}$ is deterministic in  \letkp\ 
and can be defined as 
\begin{itemize}
\item[(i)] $(z_1,z_2,z_3)\,\tilde{\land}\,(w_1,w_2,w_3) =  (z_1\sqcap w_1,z_2\sqcup w_2,(z_1 \sqcap z_3 \sqcap w_1 \sqcap w_3) \sqcup (z_2 \sqcap z_3) \sqcup (w_2 \sqcap w_3))$. 
\end{itemize} 
   Concerning  $\tilde{\lor}$, clauses (vp8)-(v12) impose   the following restrictions: 
\begin{itemize} \setl  
\item[] (vp8)-(vp9): $T \,\tilde{\lor}\, z=z \,\tilde{\lor}\, T=T$, for every $z$; 
\item[]  (vp10): $F \,\tilde{\lor}\, F=F$; 
\item[]  (vp11):  $z \,\tilde{\lor}\, w=T$ implies  that $z=T$ or $w=T$. Hence,  $\bo \,\tilde{\lor}\, \nei=\nei \,\tilde{\lor}\, \bo=T_0 \,\tilde{\lor}\, z=z \,\tilde{\lor}\, T_0=T_0$, for $z \neq T$; 
\item[]  (vp12):  $z \,\tilde{\lor}\, w=F$ implies  that $z=w=F$. Hence,  $F_0 \,\tilde{\lor}\, F_0=F_0 \,\tilde{\lor}\, F=F \,\tilde{\lor}\, F_0=F_0$.
\end{itemize}
\noi Thus, $\tilde{\lor}$ turns out to be deterministic in  \letkp\   
 and can be defined as 
\begin{itemize}
\item[(ii)] $(z_1,z_2,z_3)\,\tilde{\lor}\,(w_1,w_2,w_3)  =  (z_1\sqcup w_1,z_2\sqcap w_2, (z_2 \sqcap z_3 \sqcap w_2 \sqcap w_3) \sqcup (z_1 \sqcap z_3) \sqcup (w_1 \sqcap w_3))$. 
\end{itemize}
   With respect to  $\tilde{\to}$, clauses (vp13)-(v17) impose  the following restrictions:
\begin{itemize} \setl  
\item[]   (vp13): $F \,\tilde{\to}\, z=T$, for every $z$;
\item[] (vp14): $z \,\tilde{\to}\, T=T$, for every $z$;
\item[] (vp15): $z \,\tilde{\to}\, F=F$, for $z \in \textrm{D}$. Hence,  $T \,\tilde{\to}\, F=T_0 \,\tilde{\to}\, F=\bo \,\tilde{\to}\, F=F$;
\item[] (vp16): $z \,\tilde{\to}\, w=T$ implies that $z=F$ or $w=T$.  Hence, $\nei \,\tilde{\to}\, w=F_0 \,\tilde{\to}\, w=T_0$, for $w \neq T$; and  $z\,\tilde{\to}\, T_0=T_0$, for $z \in \textrm{D}$; 
\item[] (vp17):  $z \,\tilde{\to}\, w=F$ implies  that $w=F$. Hence, $z \,\tilde{\to}\, F_0=F_0$, for $z \in \textrm{D}$.
\end{itemize}
 \mh Thus, $\tilde{\to}$ also turns out to be deterministic, defined as 
\begin{itemize}
\item[(iii)]  $(z_1,z_2,z_3)\,\tilde{\to}\,(w_1,w_2,w_3)  =  (z_1\Rightarrow w_1,z_1\sqcap w_2, (z_1 \sqcap w_2 \sqcap w_3) \sqcup (z_2 \sqcap z_3) \sqcup (w_1 \sqcap w_3))$. 
\end{itemize}
   Given the clause (vp2), 
the third coordinate of the snapshot will not 
 be changed by the operation $\tilde{\neg}$. 
 For this reason,   
 the negation $\tilde{\neg}$ in \letkp\  also turns out to be deterministic: 
 $\tilde{\neg} T = F$, $\tilde{\neg} F = T$, $\tilde{\neg} T_0 = F_0$, and $\tilde{\neg} F_0 = T_0$.   
That is, it is defined as  
 \begin{itemize}
\item[(iv)]   $\tilde{\neg}\,(z_1,z_2,z_3) = (z_2,z_1,z_3)$.
\end{itemize} 
   Finally, the classicality operator $\tilde{\circ}$ of \letkp, by virtue of (vp1), which fixes the value of the third coordinate, 
 also becomes deterministic: for $z = T$ or $z=F$, $\tilde{\con} z = F$;        
 otherwise $\tilde{\con} z = F$.  
 That is, it is defined as:  
 \begin{itemize}
 \item[(v)] $\tilde{\circ}\,(z_1,z_2,z_3) = (z_3,\sneg z_3,1)$.
 \end{itemize}
  \end{proposition}

\mmm

The reasoning above yields the following deterministic  six-valued matrix for \letkp, obtained 
by applying the corresponding restrictions to  the six-valued Nmatrix for \letk:

\begin{definition} \ \label{def.six.val.LETKP} Let  $\mathcal{M}_6$ be the six-valued logical matrix with domain $\textsc{B}_{\letk}$, the set of designated values $\textrm{D}=\{T,T_0,\bo\}$, and the operations given by the tables below.

\mm

\begin{center}
\begin{tabular}{|c|c|c|c|c|c|c|}
\hline
 $\tilde{\wedge}$ & $T$  & $T_0$   & $\bo$ & $\nei$ & $F_0$  & $F$\\[1mm]
 \hline \hline
    $T$    & $T$  & $T_0$ & $\bo$ & $\nei$ & $F_0$ & $F$   \\[1mm] \hline
     $T_0$    & $T_0$  & $T_0$ & $\bo$ & $\nei$ & $F_0$ & $F$  \\[1mm] \hline
     $\bo$    & $\bo$  & $\bo$ & $\bo$ & $F_0$ & $F_0$ & $F$  \\[1mm] \hline
     $\nei$    & $\nei$  & $\nei$ & $F_0$ & $\nei$ & $F_0$ & $F$  \\[1mm] \hline
     $F_0$    & $F_0$  & $F_0$ & $F_0$ & $F_0$ & $F_0$ & $F$  \\[1mm] \hline
     $F$    & $F$  & $F$ & $F$ & $F$ & $F$ & $F$  \\[1mm] \hline
\end{tabular}
\hspace{1cm}
\begin{tabular}{|c|c|c|c|c|c|c|}
\hline
 $\tilde{\vee}$ & $T$  & $T_0$  & $\bo$ & $\nei$ & $F_0$  & $F$ \\[1mm]
 \hline \hline
    $T$    & $T$  & $T$ & $T$ & $T$ & $T$ & $T$   \\[1mm] \hline
     $T_0$    & $T$  & $T_0$ & $T_0$ & $T_0$ & $T_0$ & $T_0$   \\[1mm] \hline
     $\bo$    & $T$  & $T_0$ & $\bo$ & $T_0$ & $\bo$ & $\bo$   \\[1mm] \hline
     $\nei$    & $T$  & $T_0$ & $T_0$ & $\nei$ & $\nei$ & $\nei$   \\[1mm] \hline
     $F_0$    & $T$  & $T_0$ & $\bo$ & $\nei$ & $F_0$ & $F_0$   \\[1mm] \hline
     $F$    & $T$  & $T_0$ & $\bo$ & $\nei$ & $F_0$ & $F$  \\[1mm] \hline
\end{tabular}

\

\

\begin{tabular}{|c|c|c|c|c|c|c|}
\hline
 $\tilde{\to}$ & $T$  & $T_0$  & $\bo$ & $\nei$ & $F_0$  & $F$ \\[1mm]
 \hline \hline
    $T$    & $T$  & $T_0$ & $\bo$ & $\nei$ & $F_0$ & $F$   \\[1mm] \hline
     $T_0$    & $T$  & $T_0$ & $\bo$ & $\nei$ & $F_0$ & $F$  \\[1mm] \hline
     $\bo$    & $T$  & $T_0$ & $\bo$ & $\nei$ & $F_0$ & $F$  \\[1mm] \hline
     $\nei$    & $T$  & $T_0$ & $T_0$ & $T_0$ & $T_0$ & $T_0$  \\[1mm] \hline
     $F_0$    & $T$  & $T_0$ & $T_0$ & $T_0$ & $T_0$ & $T_0$  \\[1mm] \hline
     $F$    & $T$  & $T$ & $T$ & $T$ & $T$ & $T$  \\[1mm] \hline
\end{tabular}
\hspace{1.5cm}
\begin{tabular}{|c||c|} \hline
$\quad$ & $\tilde{\neg}$ \\[1mm]
 \hline \hline
    $T$   & $F$    \\[1mm] \hline
     $T_0$   & $F_0$    \\[1mm] \hline
     $\bo$   &$\bo$    \\[1mm] \hline
     $\nei$   & $\nei$    \\[1mm] \hline
     $F_0$   & $T_0$    \\[1mm] \hline
     $F$   & $T$    \\[1mm] \hline
\end{tabular}
\hspace{1.5cm}
\begin{tabular}{|c||c|}
\hline
 $\quad$ & $\tilde{\circ}$ \\[1mm]
 \hline \hline
    $T$   & $T$    \\[1mm] \hline
     $T_0$   & $F$     \\[1mm] \hline
     $\bo$   & $F$     \\[1mm] \hline
     $\nei$   & $F$     \\[1mm] \hline
     $F_0$   & $F$     \\[1mm] \hline
     $F$   & $T$    \\[1mm] \hline
\end{tabular}
\end{center}
\end{definition}

\mm

\begin{remark} \label{twist-LETKP} \  \\[1mm]
(i) Let $\mathcal{A}_{\letk}$ be the six-valued multialgebra underlying the Nmatrix  $\mathcal{M}_{\letk}$ for \letk\ 
and $\mathcal{A}_6$  the six-valued algebra  underlying the matrix 
$\mathcal{M}_6$. 
Since every algebra is  a multialgebra in which 
each entry of each multioperator returns a singleton set, it is immediate
to see that  $\mathcal{A}_6$ is a submultialgebra of $\mathcal{A}_{\letk}$. 
Assume now that $\tilde{\#}_6$ and $\tilde{\#}_{\letk}$ denote, respectively, the interpretation of the connective $\#$
 in the multialgebras $\mathcal{A}_6$ and $\mathcal{A}_{\letk}$. 
Thus,    $\tilde{\#}_6 \, z \subseteq \tilde{\#}_{\letk}\,z$ and $z\,\tilde{\#'}_6\,w \subseteq z\,\tilde{\#'}_{\letk}\,w$ 
for every $z,w \in \textsc{B}_{\letk}$, $\# \in \{\neg,\cons\}$ and $\#' \in \{\land,\lor,\to\}$. 
Therefore, 
any valuation over the matrix $\mathcal{M}_6$ is a valuation over the Nmatrix $\mathcal{M}_{\letk}$.

\m\mh (ii)  Recall from Definition~\ref{propag-rules} that   $A^T \defi \con A \land A$ and $A^F \defi \con A \land \neg A$, for every formula $A$. 
Accordingly, let   $z^T \defi \tilde{\con} z \ \tilde{\land} \ z$ and  $z^F \defi \tilde{\con} z \ \tilde{\land} \ \tilde{\neg} z$, 
for every $z \in B_{\letk}$, where the operations correspond to  $\mathcal{M}_6$. 
 It is easy to see that, for every  $z \in B_{\letk}$: $z^T = T$ if $z=T$, and $z^T=F$ otherwise; 
 and  $z^F = T$ if $z=F$, and $z^F=F$ otherwise.  
Clearly, for every valuation $v$  over the logical matrix  $\mathcal{M}_6$, $v(A^T) = v(A)^T$ and $v(A^F) = v(A)^F$. 
\end{remark}

\subsection{Soundness and completeness of the six-valued semantics of \letkp}
The next task is to prove the soundness and completeness   of \letkp\ w.r.t. the semantics given by $\mathcal{M}_6$. 
The proof is obtained by adapting  to \letkp\ the proof of Theorems~\ref{sound-Nmat-LETK} and~\ref{comp-Nmat-LETK}. 
The definition of $F$-saturated sets in \letkp\ is analogous to the one for \letk\ (recall Definition~\ref{def-F-sat-LETK}).

\begin{proposition} \label{lemma-sound-LETKP}
For every valuation $v$  over the matrix $\mathcal{M}_{6}$  the mapping $\rho_v:For(\Sigma) \to {\bf2}$ given by $\rho_v(A)=v(A)_1$ is a bivaluation for \letkp\ such that: $\rho_v(A)=1$ \ iff \ $v(A) \in \textrm{D}$, for every formula $A$.
\end{proposition}
\begin{proof}  As noted in Remark~\ref{twist-LETKP}(i), any valuation over the matrix $\mathcal{M}_{6}$ is a valuation over the Nmatrix $\mathcal{M}_{\letk}$. Thus, given a valuation $v$  over $\mathcal{M}_{6}$, the function $\rho_v$ defined as above is a bivaluation for \letk, by Proposition~\ref{lemma-sound-LETK}. It remains to prove that $\rho_v$ satisfies clauses (vp1)-(vp17) from Definition~\ref{val-sem-propag}. Thus, let $A,B \in For(\Sigma)$. Then $\rho_v(\cons\cons A)=v(\cons A)_3=(\tilde{\cons}\,v(A))_3=1$, by definition of $\mathcal{M}_{6}$. This shows that $\rho_v$ satisfies (vp1). Using again the definition of $\mathcal{M}_{6}$, $\rho_v(\cons\neg A)=v(\neg A)_3=(\tilde{\neg}\,v(A))_3=v(A)_3=\rho_v(\cons A)$. 
In order to prove that $\rho_v$ satisfies (vp3), suppose that $\rho_v(\cons A)=\rho_v(A)=1$ and  $\rho_v(\cons B)=\rho_v(B)=1$. Then $v(A)_3=v(A)_1=1$ and  $v(B)_3=v(B)_1=1$, that is, $v(A)=v(B)=T$. Hence, $v(A \land B)=v(A)\,\tilde{\land}\,v(B)=T\,\tilde{\land}\,T=T$. From this, $\rho_v(\cons(A \land B))=v(A \land B)_3=1$. For (vp4), suppose that $\rho_v(\cons A)=\rho_v(\neg A)=1$. Then $v(A)_3=v(A)_2=1$, that is, $v(A)=F$. Hence, $v(A \land B)=v(A)\,\tilde{\land}\,v(B)=F\,\tilde{\land}\,v(B)=F$, for every $B$. From this, $\rho_v(\cons(A \land B))=v(A \land B)_3=1$. Analogously it is proved that $\rho_v$ satisfies (vp5). For (vp7), suppose that $\rho_v(\cons(A \land B))=1$, and either $\rho_v(\neg A)=1$ or $\rho_v(\neg B)=1$. Then, $v(A \land B)_3=1$ and $v(A)_2=1$ or $v(B)_2=1$, hence $v(A \land B)_2=1$. Then, $v(A \land B)=F$ and so either $v(A)=F$ or $v(B)=F$. That is, either $v(A)_3=v(A)_2=1$ or $v(B)_3=v(B)_2=1$. Hence,  $\rho_v(\cons A) = \rho_v(\neg A)=1$ or  $\rho_v(\cons B)=\rho_v(\neg B)=1$ and so $\rho_v$ satisfies (vp7).  For (vp8), suppose that $\rho_v(\cons A)=\rho_v(A)=1$. Then, $v(A)_3=v(A)_1=1$, that is, $v(A)=T$. Then, $v(A \vee B)=T$ and so $v(A \lor B)_3=1$, for any $B$. That is,  $\rho_v(\cons(A \lor B))=1$, for any $B$, and so $\rho_v$ satisfies (vp8). Clause (vp9) is proved analogously. For (vp16), assume that $\rho_v(\cons(A \to B))=1$, and either $\rho_v(A)=0$ or $\rho_v(B)=1$. Then, $v(A \to B)_3=1$, and either $v(A)_1=0$ or $v(B)_1=1$, that is, $v(A \to B)_1=1$. This means that $v(A \to B)=T$, and so either $v(A)=F$ or $v(B)=T$. That is, either $v(A)_3 =v(A)_2=1$ or  $v(B)_3=v(B)_1=1$. This means that either $\rho_v(\cons A) =\rho_v(\neg A)=1$ or  $\rho_v(\cons B)=\rho_v(B)=1$. Hence, $\rho_v$ satisfies clause (vp16). 
The rest of the clauses are proved by similar arguments. 
This shows that $\rho_v$ is a bivaluation for \letkp\ such that, by definition,  $\rho_v(A)=1$ \ iff \ $v(A) \in \textrm{D}$, for every formula $A$.
\end{proof}

\begin{theorem}  [Soundness of \letkp\ w.r.t. the six-valued logical matrix $\mathcal{M}_6$] \label{sound-LETKP-matrix}  \ \\
For every set of formulas $\Gamma \cup \{A\} \subseteq For(\Sigma)$: 
$\Gamma \vdash_{\letkp} A$ \ implies that \  $\Gamma\models_{\mathcal{M}_6} A$.
\end{theorem}
\begin{proof} 
Assume that $\Gamma \vdash_{\letkp} A$. By Theorem~\ref{adeq-LETKP-bival},  $\Gamma\models_{\letkp}^2 A$. Now, let $v$ be a valuation  over the matrix $\mathcal{M}_{6}$ such that $v(B)  \in \textrm{D}$ for every $B \in \Gamma$, and let $\rho_v$ be the bivaluation for \letkp\ defined from $v$ as in Proposition~\ref{lemma-sound-LETKP}. Hence $\rho_v(B)=1$ for every $B \in \Gamma$ and so $\rho_v(A)=1$, since $\Gamma\models_{\letkp}^2 A$. From this it follows that $v(A) \in \textrm{D}$. Therefore, $\Gamma\models_{\mathcal{M}_{6}} A$.
\end{proof}

\begin{remark} \label{clause-refor}
It is worth noting that clauses (vp6), (vp7), (vp11), (vp12), (vp16) and (vp17) are equivalent, by contraposition, to the following ones:\\[1mm]
$\begin{array}{ll}
\mbox{(vp6)'} & \mbox{If $\rho(A)=\rho(B)=1$ and  either $\rho(\cons A) = 0$ or  $\rho(\cons B)=0$, then  $\rho(\cons(A \land B))=0$};\\[1mm]
\mbox{(vp7)'} & \mbox{If either $\rho(\neg A)=1$ or $\rho(\neg B)=1$; either $\rho(\cons A) = 0$ or $\rho(\neg A)=0$;}\\[1mm]
& \mbox{and either  $\rho(\cons B) = 0$ or $\rho(\neg B)=0$, then $\rho(\cons(A \land B))=0$};\\[1mm]
\mbox{(vp11)'} & \mbox{If either $\rho(A)=1$ or $\rho(B)=1$; either $\rho(\cons A) = 0$ or $\rho(A)=0$;}\\[1mm]
& \mbox{and either  $\rho(\cons B) = 0$ or $\rho(B)=0$, then $\rho(\cons(A \lor B))=0$};\\[1mm]
\mbox{(vp12)'} & \mbox{If $\rho(A)=\rho(B)=0$ and  either $\rho(\cons A) = 0$ or  $\rho(\cons B)=0$, then  $\rho(\cons(A \lor B))=0$};\\[1mm]
\mbox{(vp16)'} & \mbox{If either $\rho(A)=0$ or $\rho(B)=1$; either $\rho(\cons A) = 0$ or $\rho(\neg A)=0$;}\\[1mm]
& \mbox{and either  $\rho(\cons B) = 0$ or $\rho(B)=0$, then $\rho(\cons(A \to B))=0$};\\[1mm]
\mbox{(vp17)'} & \mbox{If $\rho(A)=\rho(\neg B)=1$ and $\rho(\cons B)=0$, then  $\rho(\cons(A \to B))=0$}.
\end{array}$\\[1mm]
This reformulation of the above-mentioned clauses will be useful for the  proof of completeness of \letkp\ w.r.t. $\mathcal{M}_6$. Moreover, in Proposition~\ref{charact-bival-LETKP} a more compact characterization of bivaluations for \letkp\ will be given.
\end{remark}

\begin{proposition} \label{lemma-comple-LETKP}
For every bivaluation $\rho$ for \letkp\  the mapping $v_{\rho}:For(\Sigma) \to \textsc{B}_{\letk}$ given by $v_\rho(A)=(\rho(A),\rho(\neg A),\rho(\circ A))$ is a valuation over the matrix $\mathcal{M}_6$ such that: $v_\rho(A) \in \textrm{D}$ \ iff \ $\rho(A)=1$, for every formula $A$.
\end{proposition}
\begin{proof} It is clear that $v_\rho(A) \in \textsc{B}_{\letk}$, hence the function is well-defined. Let us prove now that $v_\rho$ is a valuation over $\mathcal{M}_6$. Thus, let $A,B \in For(\Sigma)$ (in what follows, recall clauses (v1)-(v8) and (vp1)-(vp17) from Definitions~\ref{def-val-LETK} and~\ref{val-sem-propag}), as well as clauses (vp6)', (vp7)', (vp11)', (vp12)', (vp16)' and (vp17)' from Remark~\ref{clause-refor}).

\m\mh {\em Conjunction:} Suppose that $v_\rho(A)=v_\rho(B)=T$. Then, $\rho(\cons A)=\rho(A)=1$ and $\rho(\cons B)=\rho(B)=1$ and so, by (vp3) and (v1), $\rho(\cons(A \land B))=\rho(A \land B)=1$. Hence, $v_\rho(A\land B)=T=v_\rho(A) \,\tilde{\land}\, v_\rho(B)$. Now, suppose that $v_\rho(A)=T_0$ and $v_\rho(B) \in \{T,T_0\}$. Then, $\rho(A)=1$, $\rho(\neg A)=\rho(\cons A)=0$, $\rho(B)=1$ and $\rho(\neg B)=0$. By (v1), (v5) and (vp6)', $\rho(A \land B)=1$, $\rho(\neg(A \land B))=0$ and $\rho(\cons(A \land B))=0$. That is, $v_\rho(A\land B)=T_0=v_\rho(A) \,\tilde{\land}\, v_\rho(B)$. Analogously we prove that, if $v_\rho(B)=T_0$ and $v_\rho(A) \in \{T,T_0\}$ then $v_\rho(A\land B)=T_0=v_\rho(A) \,\tilde{\land}\, v_\rho(B)$. 
Assume now that $v_\rho(A)=\bo$ and $v_\rho(B) \in \textrm{D}$. Then, $\rho(A)=\rho(\neg A)=1$, $\rho(\cons A)=0$ and $\rho(B)=1$. By (v1), (v5) and (v8), $\rho(A \land B)=\rho(\neg(A \land B))=1$ and $\rho(\cons(A \land B))=0$. That is, $v_\rho(A\land B)=\bo=v_\rho(A) \,\tilde{\land}\, v_\rho(B)$. Analogously we prove that, if $v_\rho(B)=\bo$ and $v_\rho(A) \in \textrm{D}$ then $v_\rho(A\land B)=\bo=v_\rho(A) \,\tilde{\land}\, v_\rho(B)$. 
Now, assume that $v_\rho(A)=\nei$ and $v_\rho(B) \in \{T,T_0,\nei\}$. Then, $\rho(A)=\rho(\neg A)=\rho(\cons A)=0$ and $\rho(\neg B)=0$. Hence, by (v1) and (v5) and (v8), $\rho(A \land B)=\rho(\neg(A \land B))=\rho(\cons(A \land B))=0$. That is, $v_\rho(A\land B)=\nei=v_\rho(A) \,\tilde{\land}\, v_\rho(B)$. Analogously we prove that, if $v_\rho(B)=\nei$ and $v_\rho(A) \in \{T,T_0\}$ then $v_\rho(A\land B)=\nei=v_\rho(A) \,\tilde{\land}\, v_\rho(B)$. Suppose now that $v_\rho(A)=F_0$ and $v_\rho(B)  \neq F$. Then $\rho(A)=0$, $\rho(\neg A)=1$, $\rho(\cons A)=0$ and: either $\rho(\neg B)=0$ or $\rho(\cons B)=0$. Hence, by (v1), (v5) and (vp7)', $\rho(A \land B)=0$, $\rho(\neg(A \land B))=1$  and $\rho(\cons(A \land B))=0$. That is,  $v_\rho(A\land B)=F_0=v_\rho(A) \,\tilde{\land}\, v_\rho(B)$. Analogously we prove that, if $v_\rho(B)=F_0$ and $v_\rho(A) \neq F$ then $v_\rho(A\land B)=F_0=v_\rho(A) \,\tilde{\land}\, v_\rho(B)$.  Assume now that $v_\rho(A)=\bo$ and $v_\rho(B)=\nei$. Then $\rho(A)=\rho(\neg A)=1$, $\rho(\cons A)=0$ and $\rho(B)=\rho(\neg B)=\rho(\cons B)=0$. Thus, by (v1), (v5) and (vp7)', $\rho(A \land B)=0$, $\rho(\neg(A \land B))=1$  and $\rho(\cons(A \land B))=0$. That is,  $v_\rho(A\land B)=F_0=v_\rho(A) \,\tilde{\land}\, v_\rho(B)$. Analogously we prove that, if $v_\rho(B)=\nei$ and $v_\rho(A)=\bo$ then $v_\rho(A\land B)=F_0=v_\rho(A) \,\tilde{\land}\, v_\rho(B)$. Finally, suppose that either $v_\rho(A)=F$ or $v_\rho(B)=F$. Then, either  $\rho(\cons A)=\rho(\neg A)=1$ or $\rho(\cons B)=\rho(\neg B)=1$ and so, by (vp4), (vp5) and (v5), $\rho(\cons(A \land B))=\rho(\neg(A \land B))=1$. Therefore $v_\rho(A\land B)=F=v_\rho(A) \,\tilde{\land}\, v_\rho(B)$.  

\m\mh{\em Disjunction:} Suppose that  either $v_\rho(A)=T$ or $v_\rho(B)=T$. Then, either  $\rho(\cons A)=\rho(A)=1$ or $\rho(\cons B)=\rho(B)=1$ and so, by (vp8), (vp9) and (v2), $\rho(\cons(A \lor B))=\rho(A \lor B)=1$. Therefore $v_\rho(A\lor B)=T=v_\rho(A) \,\tilde{\lor}\, v_\rho(B)$. Suppose now that $v_\rho(A)=T_0$ and $v_\rho(B)  \neq T$. Then $\rho(A)=1$, $\rho(\neg A)=0$, $\rho(\cons A)=0$ and: either $\rho(B)=0$ or  $\rho(\cons B)=0$. By (v2),(v6) and  (vp11)' it follows that  $\rho(A \lor B)=1$, $\rho(\neg(A \lor B))=0$ and $\rho(\cons(A \lor B))=0$. That is,  $v_\rho(A\lor B)=T_0=v_\rho(A) \,\tilde{\lor}\, v_\rho(B)$. Analogously we prove that, if $v_\rho(B)=T_0$ and $v_\rho(A) \neq T$ then $v_\rho(A\lor B)=T_0=v_\rho(A) \,\tilde{\lor}\, v_\rho(B)$. Now, assume that $v_\rho(A)=\bo$ and $v_\rho(B) \in \{\bo,F_0,F\}$. Then, $\rho(A)=\rho(\neg A)=1$, $\rho(\cons A)=0$ and $\rho(\neg B)=1$. Hence, by (v2), (v6) and (v8), $\rho(A \lor B)=\rho(\neg(A \lor B))=1$ and $\rho(\cons(A \lor B))=0$. That is, $v_\rho(A\lor B)=\bo=v_\rho(A) \,\tilde{\lor}\, v_\rho(B)$. Analogously we prove that, if $v_\rho(B)=\bo$ and $v_\rho(A) \in \{F,F_0\}$ then $v_\rho(A\lor B)=\bo=v_\rho(A) \,\tilde{\lor}\, v_\rho(B)$. Suppose now that $v_\rho(A)=\bo$ and $v_\rho(B)=\nei$. Then $\rho(A)=\rho(\neg A)=1$, $\rho(\cons A)=0$ and $\rho(B)=\rho(\neg B)=\rho(\cons B)=0$. Then, by (v2), (v6) and (vp11)', $\rho(A \lor B)=1$, $\rho(\neg(A \lor B))=0$  and $\rho(\cons(A \lor B))=0$. That is,  $v_\rho(A\lor B)=T_0=v_\rho(A) \,\tilde{\lor}\, v_\rho(B)$. Analogously we prove that, if $v_\rho(B)=\nei$ and $v_\rho(A)=\bo$ then $v_\rho(A\lor B)=T_0=v_\rho(A) \,\tilde{\lor}\, v_\rho(B)$. 
Assume now that $v_\rho(A)=\nei$ and $v_\rho(B) \in \textrm{ND}$. Then, $\rho(A)=\rho(\neg A)=\rho(\cons A)=0$ and $\rho(B)=0$. By (v2), (v6) and (v8), $\rho(A \lor B)=\rho(\neg(A \lor B))=\rho(\cons(A \lor B))=0$. That is, $v_\rho(A\lor B)=\nei=v_\rho(A) \,\tilde{\lor}\, v_\rho(B)$. Analogously we prove that, if $v_\rho(B)=\nei$ and $v_\rho(A) \in \textrm{ND}$ then $v_\rho(A\lor B)=\nei=v_\rho(A) \,\tilde{\lor}\, v_\rho(B)$. Now, suppose that  $v_\rho(A)=F_0$ and $v_\rho(B) \in \{F,F_0\}$. Then, $\rho(A)=0$, $\rho(\neg A)=1$, $\rho(\cons A)=0$, $\rho(B)=0$ and $\rho(\neg B)=1$. By (v2), (v6) and (vp12)', $\rho(A \lor B)=1$, $\rho(\neg(A \lor B))=0$ and $\rho(\cons(A \lor B))=0$. That is, $v_\rho(A\lor B)=F_0=v_\rho(A) \,\tilde{\lor}\, v_\rho(B)$. Analogously we prove that, if $v_\rho(B)=F_0$ and $v_\rho(A) \in \{F,F_0\}$ then $v_\rho(A\lor B)=F_0=v_\rho(A) \,\tilde{\lor}\, v_\rho(B)$. Finally, suppose that $v_\rho(A)=v_\rho(B)=F$. Then, $\rho(A)=0$, $\rho(\cons A)=\rho(\neg A)=1$, $\rho(B)=0$ and $\rho(\cons B)=\rho(\neg B)=1$ and so, by (vp10) and (v2), $\rho(\cons(A \lor B))=1$ and $\rho(A \lor B)=1$. Hence, $v_\rho(A\lor B)=F=v_\rho(A) \,\tilde{\lor}\, v_\rho(B)$. 

\m\mh{\em Implication:} Suppose that $v_\rho(A)=F$. Then, $\rho(A)=0$, $\rho(\neg A)=1$ and $\rho(\cons A)=1$. Using (v3), (v7) and (vp13) it follows that, for any $B$, $\rho(A \to B)=1$, $\rho(\neg(A \to B))=0$ and $\rho(\cons(A \to B))=1$. That is, $v_\rho(A \to B)=T=v_\rho(A) \,\tilde{\to}\, v_\rho(B)$. Analogously (but now by using (vp14)) it is proven that, if  $v_\rho(B)=T$, $v_\rho(A \to B)=T=v_\rho(A) \,\tilde{\to}\, v_\rho(B)$ for any $A$.  Suppose now that $v_\rho(A)=F_0$ and $v_\rho(B)  \neq T$. Then $\rho(A)=0$, $\rho(\neg A)=1$, $\rho(\cons A)=0$ and: either $\rho(B)=0$ or  $\rho(\cons B)=0$. By (v3),(v7) and  (vp16)' it follows that  $\rho(A \to B)=1$, $\rho(\neg(A \to B))=0$ and $\rho(\cons(A \to B))=0$. That is,  $v_\rho(A\to B)=T_0=v_\rho(A) \,\tilde{\to}\, v_\rho(B)$. Analogously we prove that, if $v_\rho(B)=T_0$ and $v_\rho(A) \neq F$ then $v_\rho(A\to B)=T_0=v_\rho(A) \,\tilde{\to}\, v_\rho(B)$.
Assume now that $v_\rho(B)=\bo$ and $v_\rho(A) \in \textrm{D}$. Then, $\rho(B)=\rho(\neg B)=1$, $\rho(\cons B)=0$ and $\rho(A)=1$. By (v3), (v7) and (v8), $\rho(A \to B)=\rho(\neg(A \to B))=1$ and $\rho(\cons(A \to B))=0$. That is, $v_\rho(A\to B)=\bo=v_\rho(A) \,\tilde{\to}\, v_\rho(B)$. Analogously we prove that, if $v_\rho(B)=\nei$ and $v_\rho(A) \in \textrm{D}$ then $v_\rho(A\to B)=\nei=v_\rho(A) \,\tilde{\to}\, v_\rho(B)$. 
Now, assume that $v_\rho(B)=F_0$ and $v_\rho(A) \in \textrm{D}$. Then, $\rho(B)=0$, $\rho(\neg B)=1$, $\rho(\cons B)=0$ and $\rho(A)=1$. By (v3), (v7) and (vp17)', $\rho(A \to B)=0$, $\rho(\neg(A \to B))=1$ and $\rho(\cons(A \to B))=0$. That is, $v_\rho(A\to B)=F_0=v_\rho(A) \,\tilde{\to}\, v_\rho(B)$. Suppose now that $v_\rho(B)=F$ and $v_\rho(A) \in \textrm{D}$.  Then, $\rho(B)=0$, $\rho(\neg B)=1$, $\rho(\cons B)=1$ and $\rho(A)=1$. By (v3), (v7) and (vp15), $\rho(A \to B)=0$, $\rho(\neg(A \to B))=1$ and $\rho(\cons(A \to B))=1$. That is, $v_\rho(A\to B)=F=v_\rho(A) \,\tilde{\to}\, v_\rho(B)$.
Finally, suppose that $v_\rho(A)=\nei$ and $v_\rho(B) \in \{\bo,\nei,F_0,F\}$. Then, $\rho(A)=\rho(\neg A)=\rho(\cons A)=0$ and: either $\rho(\cons B)=0$ or $\rho(B)=0$. By (v3), (v7) and (vp16)', $\rho(A \to B)=0$, $\rho(\neg(A \to B))=1$ and $\rho(\cons(A \to B))=0$. That is, $v_\rho(A\to B)=T_0=v_\rho(A) \,\tilde{\to}\, v_\rho(B)$.

\m\mh{\em Negation:}
Let $v_\rho(A)=(\rho(A),\rho(\neg A),\rho(\cons A))$. Then, 
by (v4) and (vp2): 
	$$\begin{array}{lll}
	v_\rho(\neg A) &=&  (\rho(\neg A),\rho(\neg\neg A),\rho(\cons\neg A))\\
	&=& (\rho(\neg A),\rho(A),\rho(\cons A))= \tilde{\neg}\, v_\rho(A).
	\end{array}$$

\m\mh{\em Classicality:}
Let $v_\rho(A)=(\rho(A),\rho(\neg A),\rho(\cons A))$. 
Then, by (vp1):
	$$\begin{array}{lll}
	v_\rho(\cons A) &=&  (\rho(\cons A),\rho(\neg \cons A),\rho(\cons\cons A))\\
	&=& (\rho(\cons A),\sneg \rho(\cons A),1)= \tilde{\cons}\, v_\rho(A).
	\end{array}$$ \ 
	
\m \mh This shows that $v_{\rho}$ is a valuation over the matrix $\mathcal{M}_{6}$ such that, 
for every formula $A$, $v_\rho(A) \in \textrm{D}$ \ iff \ $\rho(A)=1$.
\end{proof}

\mm

\begin{theorem}  [Completeness of \letkp\ w.r.t. the six-valued logical matrix $\mathcal{M}_6$] \label{comple-LETKP-matrix}  \ \\
For every set of formulas $\Gamma \cup \{A\} \subseteq For(\Sigma)$: 
$\Gamma\models_{\mathcal{M}_6} A$ \ implies that \ $\Gamma \vdash_{\letkp} A$.
\begin{proof} 
Assume that $\Gamma\models_{\mathcal{M}_{6}} A$, and let $\rho$ be a bivaluation for \letkp\ such that $\rho(B)=1$ for every $B \in \Gamma$. Let $v_{\rho}$ be defined as in Proposition~\ref{lemma-comple-LETKP}. Then, $v_{\rho}$ is a valuation over $\mathcal{M}_{6}$ such that $v_{\rho}(B) \in \textrm{D}$, for every $B \in \Gamma$. By hypothesis, $v_{\rho}(A) \in \textrm{D}$, whence $\rho(A)=1$. This shows that  $\Gamma\models_{\letkp}^2 A$. By Theorem~\ref{adeq-LETKP-bival}, $\Gamma \vdash_{\letkp} A$.
\end{proof}\end{theorem}

\noindent As announced before, the definition of bivaluations for \letkp\ can be drastically simplified in terms of Boolean operators:
 
\begin{proposition} \label{charact-bival-LETKP} Let $\rho$ be a bivaluation for \letk. Then, $\rho$ is a bivaluation for \letkp\ iff it satisfies, in addition,  (vp1) and (vp2) (from Definition~\ref{val-sem-propag}) plus the following properties, expressed in the language of Boolean algebras:\\[1mm]
$\begin{array}{ll}
\mbox{(v9)} & \rho(\cons(A \land B))= a \sqcup b \sqcup c, \mbox{ where}\\[1mm]
& a=\rho(A) \sqcap \rho(\cons A) \sqcap \rho(B) \sqcap \rho(\cons B), \  b=\rho(\neg A) \sqcap \rho(\cons A), \ c=\rho(\neg B) \sqcap \rho(\cons B);\\[1mm]
\mbox{(v10)} &  \rho(\cons(A \lor B))= a' \sqcup b' \sqcup c', \mbox{ where}\\[1mm]
& a'=\rho(\neg A) \sqcap \rho(\cons A) \sqcap \rho(\neg B) \sqcap \rho(\cons B), \  b'=\rho(A) \sqcap \rho(\cons A), \ c'=\rho(B) \sqcap \rho(\cons B);\\[1mm]
\mbox{(v11)} &  \rho(\cons(A \to B))= a'' \sqcup b'' \sqcup c'', \mbox{ where}\\[1mm]
& a''=\rho(A) \sqcap \rho(\neg B) \sqcap \rho(\cons B), \ b''=\rho(\neg A) \sqcap \rho(\cons A), \  c''=\rho(B) \sqcap \rho(\cons B).
\end{array}$
\begin{proof} It follows by a tedious but straightforward verification.
\end{proof}\end{proposition}

\subsection{\letkp\ is Blok-Pigozzi algebraizable} \label{sect-BPalg}

In this section we show that \letkp\ has enough expressive power  to be algebraizable in the general sense proposed by 
 Blok and  Pigozzi in \cite{blok:pig:89} (see also  \cite{Font:book}). 

First, note that  in \letkp\  a bi-implication $A \toot B $ defined as $ (A \to B) \land (B \to A)$ does not preserve logical equivalence through the connectives: for instance, $T_0 \toot \bo$ gets the designated value $\bo$, but $\neg T_0 \toot \neg \bo= F_0 \toot \bo=F_0$, which is non-designated. In \nel\ and $\nel^\bot$ an `equivalence' operator which preserves logical  equivalence through the connectives (that is, defines a logical congruence) is defined as follows: $A \Leftrightarrow B \defi (A \toot B) \land (\neg A \toot \neg B)$. In terms of its (two-dimensional) twist structures, this means that two pairs are equal when the respective coordinates coincide. But in \letkp\ we are dealing with three-dimensional twist structures, that is, the snapshots are triples instead of pairs. Since we have twist operators to `read' each coordinate of the snapshots ($\tilde{\neg}$ `reads' the second coordinate, while $\tilde{\cons}$ `reads' the third one) an appropriate  notion of  `equivalence' (representing identity between triples) in \letkp\ should be the following:
$$A \equiv B \defi (A \toot B) \land (\neg A \toot \neg B) \land (\cons A \toot \cons B).$$
The  table of the interpretation $\tilde{\equiv}$ of this  connective in $\mathcal{M}_6$ is as follows: 

\begin{center}
\begin{tabular}{|c|c|c|c|c|c|c|}
\hline
 $\tilde{\equiv}$ & $T$  & $T_0$  & $\bo$ & $\nei$ & $F_0$  & $F$ \\[1mm]
 \hline \hline
    $T$    & $T$  & $F$ & $F$ & $F$ & $F$ & $F$   \\[1mm] \hline
     $T_0$    & $F$  & $T_0$ & $F_0$ & $\nei$ & $F_0$ & $F$  \\[1mm] \hline
     $\bo$    & $F$  & $F_0$ & $\bo$ & $\nei$ & $F_0$ & $F$  \\[1mm] \hline
     $\nei$    & $F$  & $\nei$ & $\nei$ & $T_0$ & $\nei$ & $F$  \\[1mm] \hline
     $F_0$    & $F$  & $F_0$ & $F_0$ & $\nei$ & $T_0$ & $F$  \\[1mm] \hline
     $F$    & $F$  & $F$ & $F$ & $F$ & $F$ & $T$  \\[1mm] \hline
\end{tabular}
\end{center}

\

\noindent
From the  table above, it is immediate to prove the following relevant properties of $\equiv$:

\begin{proposition} \label{prop-equivBP}
The following properties hold in  $\mathcal{M}_6$:\\[1mm]
(1) \ For every $z,w \in \textsc{B}_{\letk}$, $(z \,\tilde{\equiv}\, w) \in \textrm{D}$ \ iff \ $z=w$.\\[1mm]
(2) \ For every formulas $A$ and $B$, and for every valuation $v$ over $\mathcal{M}_6$: 
$v(A \equiv B)  \in \textrm{D} \ \mbox{ iff } \ v(A)=v(B).$ \\[1mm] 
(3) \ $\models_{\mathcal{M}_6} (A \equiv A)$ \ for every formula $A$.\\[1mm]
(4) \ $(A \equiv B) \models_{\mathcal{M}_6} (B \equiv A)$ for every formulas $A$ and $B$. \\[1mm]
(5) \ $(A \equiv B), (B \equiv C) \models_{\mathcal{M}_6} (A \equiv C)$ for every formulas $A$, $B$ and $C$.\\[1mm]
(6) \ $(A \equiv B) \models_{\mathcal{M}_6} (\#\, A \equiv \#\, B)$ for every formulas $A$ and $B$ and $\# \in \{\neg,\cons\}$. \\[1mm]
(7) \ $(A \equiv B), (C \equiv D) \models_{\mathcal{M}_6} (A\,\#\,C \equiv B\,\#\,D)$ for every formulas $A$, $B$, $C$ and $D$ and $\# \in \{\land,\lor,\to\}$.\\[1mm]
(8) \ $(A \equiv (A \to A))  \models_{\mathcal{M}_6} A$, and $A  \models_{\mathcal{M}_6} (A \equiv (A \to A))$  \ for every formula $A$.
\begin{proof}
Item (1) is immediate from the truth-table for $\tilde{\equiv}$ displayed above. An analytical proof can be done as follows: For every $z,w \in \textsc{B}_{\letk}$ let  $z \,\tilde{\leftrightarrow}\, w \defin (z \,\tilde{\to}\, w) \,\tilde{\land}\, (w \,\tilde{\to}\, z)$. Then, $(z \,\tilde{\leftrightarrow}\, w)_1=(z_1 \Rightarrow w_1) \sqcap (w_1 \Rightarrow z_1)$ and so $z \,\tilde{\leftrightarrow}\, w \in \textrm{D}$ iff $z_1=w_1$.  From this, 
$\tilde{\neg}\,z \,\tilde{\leftrightarrow}\, \tilde{\neg}\,w \in \textrm{D}$ iff $z_2=w_2$, and $\tilde{\cons}\,z \,\tilde{\leftrightarrow}\, \tilde{\cons}\,w \in \textrm{D}$ iff $z_3=w_3$. This means that $(z \,\tilde{\equiv}\, w) \in \textrm{D}$ \ iff $z_i=w_i$ for $i=1,2,3$, iff $z=w$. Item~(2) follows from~(1) and from the fact that $v(A \equiv B)= (v(A) \,\tilde{\equiv}\, v(B))$ for every formulas $A$ and $B$, and every valuation $v$  over $\mathcal{M}_6$. Items (3)-(7) are immediate from~(2), taking into account (for items~(6) and~(7)) that the operations in $\mathcal{M}_6$ are functional, that is, deterministic. For (8), let $v$ be a valuation over  $\mathcal{M}_6$ such that $v(A \equiv (A \to A)) \in \textrm{D}$. By~(2), $v(A)=v(A \to A)=v(A) \,\tilde{\to}\, v(A)$. But $z \,\tilde{\to}\, z\in \textrm{D}$ for every $z \in \textsc{B}_{\letk}$. Hence, $v(A) \in \textrm{D}$, showing that $(A \equiv (A \to A))  \models_{\mathcal{M}_6} A$. Conversely, let $v$  be a valuation over  $\mathcal{M}_6$ such that $v(A) \in \textrm{D}$. Hence, $v(A \to A)=v(A) \,\tilde{\to}\, v(A)=v(A)$, by definition of $\tilde{\to}$ (since  $z \,\tilde{\to}\, z=z$ for every $z\in \textrm{D}$). By~(2), $v(A \equiv (A \to A)) \in \textrm{D}$. This shows that $A  \models_{\mathcal{M}_6} (A \equiv (A \to A))$.
\end{proof}\end{proposition}

\begin{theorem} 
The logic \letkp\ is algebraizable in the sense of Blok-Pigozzi.
\begin{proof}
Let $p_1$ and $p_2$ be two different propositional variables, and consider the sets 
$$\Delta(p_1,p_2)=\{(p_1 \equiv p_2)\} \ \mbox{  and } \ E(p_1)=\{\langle p_1, (p_1 \to p_1)\rangle\}.$$
By items~(3)-(8) of Proposition~\ref{prop-equivBP}, the sets $\Delta(p_1,p_2)$ and $E(p_1)$ show that the logic \letkp, presented by means of $\mathcal{M}_6$, is algebraizable in the sense of Blok-Pigozzi. Indeed, conditions (3)-(8) of Proposition~\ref{prop-equivBP} are exactly the requirements for $\Delta(p_1,p_2)$ and $E(p_1)$ stated in Theorem~4.7 of~\cite{blok:pig:89} for a given logic being algebraizable.\footnote{Observe that conditions~(6) and~(7) of Proposition~\ref{prop-equivBP} depend on the signature of the given logic. In the terminology of~\cite{blok:pig:89}, it can be said that $\Delta(p_1,p_2)$ is a system of equivalence formulas, while $E(p_1)$ (written as $p_1 \approx (p_1 \to p_1)$) is a system of defining equations for the deductive system generated by $\mathcal{M}_6$, that is, \letkp.}
\end{proof}\end{theorem}

\mh This result, combined with the family of twist models for \letkp\ and  the relationship of \letkp\ with involutive Stone algebras, to be studied in the following two sections, opens interesting possibilities for future research of \letkp\ from the perspective of abstract algebraic logic.

\section{Twist models for \letkp\ } \label{sect-lattice}

In Section~\ref{sec.six.valued.letkp} we have seen how the restrictions imposed on the Nmatrix of \letk\ by the rules of propagation 
of classicality  yield  a (deterministic) six-valued semantics for \letkp, the matrix $\mathcal{M}_6$. 
  $\mathcal{A}_{\letk}$ is  the six-valued multialgebra underlying the Nmatrix  $\mathcal{M}_{\letk}$,  
which was presented as a  swap structure in Definition~\ref{defNmatLETK}.   
 In order to comply with the  axiom and rules of propagation of classicality of  \letkp, 
 $\mathcal{A}_{\letk}$   
becomes the algebra $\mathcal{A}_6$ for \letkp. The latter  is the underlying algebra of the  logical matrix $\mathcal{M}_6$.

Twist structures  are special cases of swap structures: while the latter can be multialgebras, 
the former are algebras,  based on 
operations instead of multioperations, and so the respective semantics are deterministic.\footnote{For a more detailed discussion 
of swap and twist structures, see Coniglio et al. \cite{con.fig.gol.2018},  in particular sections 9.3 and 9.4.} 
In what follows, the algebra $\mathcal{A}_6$ will be presented  as a 
\textit{three-dimensional twist structure}. 
We will show  how  $\mathcal{A}_6$   
can be  generalized to twist algebras generated by arbitrary Boolean algebras. 
This will produce a class of twist-valued models for \letkp, one for each Boolean algebra, 
which characterizes \letkp. 
Moreover, it will be proved that these models are, indeed, bounded lattices in which 
 suprema and infima are respectively given by the operators $\tilde{\land}$ and $\tilde{\lor}$ and 
  the top and bottom elements are given by $T$ and $F$.

\begin{definition} [Twist structures for \letkp]  \label{twist-BA} \ 

\m\mh 
Let $\mathcal{B}= \langle {\bf B}, \sqcap,\sqcup,\Rightarrow,\sneg,0,1 \rangle$ be a Boolean algebra. The {\em twist structure for \letkp\ induced by $\mathcal{B}$} is the algebra $\mathcal{T}_{\mathcal{B}}=\langle \textsc{B}_{\letk}^{\mathcal{B}}, \tilde{\land},\tilde{\lor},\tilde{\to},\tilde{\neg},\tilde{\cons}\rangle$ over $\Sigma$ such that
$$\textsc{B}_{\letk}^{\mathcal{B}}=\{z \in {\bf B}^3 \ : \ z_3 \leq z_1 \sqcup z_2 \ \mbox{ and } \ z_1 \sqcap z_2 \sqcap z_3=0  \}$$
and the operations are defined as follows:  
\begin{itemize}\setl 
\item[(i)] $(z_1,z_2,z_3)\,\tilde{\land}\,(w_1,w_2,w_3) =  (z_1\sqcap w_1,z_2\sqcup w_2,(z_1 \sqcap z_3 \sqcap w_1 \sqcap w_3) \sqcup (z_2 \sqcap z_3) \sqcup (w_2 \sqcap w_3))$,
\item[(ii)] $(z_1,z_2,z_3)\,\tilde{\lor}\,(w_1,w_2,w_3)  =  (z_1\sqcup w_1,z_2\sqcap w_2, (z_2 \sqcap z_3 \sqcap w_2 \sqcap w_3) \sqcup (z_1 \sqcap z_3) \sqcup (w_1 \sqcap w_3))$,
\item[(iii)] $(z_1,z_2,z_3)\,\tilde{\to}\,(w_1,w_2,w_3)  =  (z_1\Rightarrow w_1,z_1\sqcap w_2, (z_1 \sqcap w_2 \sqcap w_3) \sqcup (z_2 \sqcap z_3) \sqcup (w_1 \sqcap w_3))$,
\item[(iv)] $\tilde{\neg}\,(z_1,z_2,z_3)  =  (z_2,z_1,z_3)$, 
\item[(v)] $\tilde{\circ}\,(z_1,z_2,z_3)  =  (z_3,\sneg z_3,1)$.

\end{itemize}
\end{definition}

\m\mh Note  that  $\mathcal{T}_{\mathcal{B}_2}$ is exactly $\mathcal{A}_6$. Each twist structure $\mathcal{T}_{\mathcal{B}}$ for \letkp\ naturally induces a logical matrix $\mathcal{M}(\mathcal{B}) = \langle \mathcal{T}_{\mathcal{B}}, \textrm{D}_{\mathcal{B}}\rangle$ where $\textrm{D}_{\mathcal{B}}=\{z \in \textsc{B}_{\letk}^{\mathcal{B}} \ : \ z_1 =1 \}$. Let $Mat(\letkp)$ be the class of logical matrices of the form $\mathcal{M}(\mathcal{B})$, and let $\models_{Mat(\letkp)}$ be the associated consequence relation. Notice that $\mathcal{M}(\mathcal{B}_2)=\mathcal{M}_6$.

In Definition~\ref{defNmatLETK} the multioperations of the multialgebra $\mathcal{A}_{\letk}$ were presented 
 in the framework of a  non-deterministic swap structure in which the third coordinate of each snapshot of \letk\ 
 \textit{is not} functionally 
 determined from the input(s). 
 On the other hand,  the tables of \letkp\ presented in Definition~\ref{def.six.val.LETKP}   
 make it clear that in  
 $\mathcal{A}_6$, the third coordinate of snapshots, obtained by applying the respective operation, is functionally determined  
 from the input(s). Note, in addition, that a snapshot is `classical' -- that is, it belongs to $\{T,F\}$ --  
 exactly when the third coordinate is 1. 
 Thus, in the case of $\tilde{\land}$, the output $z \,\tilde{\land}\, w$ is `classical' exactly when: (i)  $z=w=T$, or (ii) $z=F$, or $w=F$. This yields  the item (i) above. 
 By   analogous reasoning for  $\tilde{\lor}$ and $\tilde{\to}$, we obtain  items (ii) and (iii) 
above.  
Given the equivalence between $\con A$ and  $\con \neg A$,  
the negation $\tilde{\neg}$ in \letkp\  is deterministic, and defined by item (iv). 
Finally, the classicality operator $\tilde{\circ}$ of \letkp,  given clause (vp1) 
of Definition~\ref{val-sem-propag},  
also becomes  deterministic, and so defined by item (v).   
Indeed, by (vp1), $\rho(\cons\cons A)=1$. This, together with Proposition~\ref{der-rules},  
 implies that $\rho(\neg \con A)=\sneg \rho(\con A)$. Hence, by (vp1), $\rho(\neg A)=\sneg \rho(A)$ and $\rho(\cons\cons A)=1$.

\begin{remark} [Decidability in \letkp\ is reduced to decidability in \cpl] \label{formulas-terms} \ 

\m\mh 
Let  $A(p_1,\ldots,p_k)$ be a formula over the signature of \letkp\ depending at most on the propositional variables $p_1,\ldots,p_k$, and let $v$ be a valuation over a matrix $\mathcal{M}(\mathcal{B})$ for \letkp. Since $v(p_i)$ is in ${\bf B}^3$, its value can be represented by 3 new propositional variables $p_i^j$, each one representing $v(p_i)_j$, the $jth$-projection of $v(p_i)$ for $1 \leq j \leq 3$. Let $\mathcal{V}_3^k=\{p_i^j \ : \ 1 \leq i \leq k$ and $1 \leq j \leq 3\}$ be the set of such new propositional variables. Hence, the formal expression for $v(A)_1$ (that is, the first coordinate of $v(A)$) can be represented by a term $\tau_A$ in $For_3^k(\Sigma_{BA})$, the language generated by  the set of variables $\mathcal{V}_3^k$ over the signature $\Sigma_{BA}=\{\sqcap,\sqcup,\Rightarrow, \sneg,\bot,\top\}$ of Boolean algebras.  For instance, given $A=(p_1 \land \neg p_2) \vee \cons p_1$ and $B=\neg ((p_1 \land \neg p_2) \vee \cons p_1)$ then $v(A)_1$ and $v(B)_1$ are represented, respectively, by the terms $\tau_A=(p_1^1 \sqcap p_2^2) \sqcup p_1^3$ and  $\tau_B= (p_1^2 \sqcup p_2^1) \sqcap \sneg p_1^3$ in $For_3^k(\Sigma_{BA})$, given that $v((p_1 \land \neg p_2) \vee \cons p_1)_1=(v(p_1)_1 \sqcap v(\neg p_2)_1) \sqcup v(\cons p_1)_1 = (v(p_1)_1 \sqcap v(p_2)_2) \sqcup v(p_1)_3$ and $v(\neg ((p_1 \land \neg p_2) \vee \cons p_1))_1=v((p_1 \land \neg p_2) \vee \cons p_1)_2 = (v(p_1)_2 \sqcup v(\neg p_2)_2) \sqcap v(\cons p_1)_2 = (v(p_1)_2 \sqcup v(p_2)_1) \sqcap \sneg v(p_1)_3$.
The fact that every $v(p_i)$ is a snapshot instead of an arbitrary  triple in ${\bf B}^3$ is represented by the term $\bar{\tau}_k=\tau_1 \sqcap \ldots \sqcap \tau_k$ such that $\tau_i=(p_i^3 \Rightarrow (p_i^1 \sqcup p_i^2)) \sqcap \sneg(p_i^1 \sqcap p_i^2 \sqcap p_i^3)$ for $1 \leq i \leq k$.

Now,  suppose that $\Gamma \models_{\mathcal{M}_6} A$ such that $\Gamma = \{A_1, \ldots A_n\}$ is non-empty and all these formulas depend on $p_1,\ldots, p_k$. Let $B=A_1 \land \ldots \land A_n$, and let $v$ be a valuation over $\mathcal{M}_6$ (that is, $v:\mathcal{V} \to \textsc{B}_{\letk}$). Then, $v(B) \in \textrm{D}$ implies that $v(A) \in \textrm{D}$ or, equivalently, $v(B)_1=1$ implies that $v(A)_1=1$.
This means that, for every  homomorphism $h:For_3^k(\Sigma_{BA}) \to \mathcal{B}_2$ such that $h(p_i^j)=v(p_i)_j$ for $1 \leq i\leq k$ and $1 \leq j\leq 3$ for a valuation $v$,  $h(\tau_B)=1$ implies that $h(\tau_A)=1$ or, equivalently, $h(\tau_B \Rightarrow \tau_A)=1$. If $h$ is defined as above from a function $v:\mathcal{V} \to {\bf 2}^3$ then  $v(p_i) \in \textsc{B}_{\letk}$ for $1 \leq i \leq k$ iff  $h(\bar{\tau}_k)=1$. Then, for every  homomorphism $h:For_3^k(\Sigma_{BA}) \to \mathcal{B}_2$, $h(\bar{\tau}_k)=1$ (i.e., $(h(p_i^1),h(p_i^2) ,h(p_i^3)) \in \textsc{B}_{\letk}$ for $1 \leq i \leq k$) implies that $h(\tau_B \Rightarrow \tau_A)=1$. Equivalently, $h(\bar{\tau}_k \Rightarrow (\tau_B \Rightarrow \tau_A))=1$ for every $h$. In other words, $\Gamma \models_{\mathcal{M}_6} A$ if and only if $\mathcal{B}_2$ validates the equation $(\bar{\tau}_k \Rightarrow (\tau_B \Rightarrow \tau_A))\approx \top$ in the language of Boolean algebras.  The later is equivalent to saying that the formula $\bar{\tau}_k \Rightarrow (\tau_B \Rightarrow \tau_A)$ is a tautology in \cpl\ (expressed in the signature $\Sigma_{BA}$).
\end{remark}

\begin{theorem} [Soundness and completeness of \letkp\ w.r.t. $Mat(\letkp)$] \label{sound-compLETKPswap} For every set of formulas $\Gamma \cup \{A\}$ over $\Sigma$: $\Gamma \vdash_{\letkp} A$ \ iff \  $\Gamma \models_{Mat(\letkp)} A$.
\begin{proof} 
(Left to right - Soundness): Suppose that $\Gamma \vdash_{\letkp} A$, and assume that $\Gamma = \{A_1, \ldots A_n\}$ is non-empty (the proof for the case $\Gamma=\emptyset$ is analogous but easier). Assume that every formula in $\Gamma \cup \{A\}$ depends at most on the propositional variables $p_1,\ldots,p_k$. By Theorem~\ref{sound-LETKP-matrix},   $\Gamma\models_{\mathcal{M}_6} A$. By Remark~\ref{formulas-terms} (and using the notation established therein) if follows that $\mathcal{B}_2$ validates the equation $(\bar{\tau}_k \Rightarrow (\tau_B \Rightarrow \tau_A))\approx \top$ in the language of Boolean algebras, where $B=A_1 \land \ldots \land A_n$. By Remark~\ref{val-Boole}(2), any Boolean algebra $\mathcal{B}$ validates the equation $(\bar{\tau}_k \Rightarrow (\tau_B \Rightarrow \tau_A))\approx \top$. That is, for every homomorphism $h:For_3^k(\Sigma_{BA}) \to \mathcal{B}$, $h(\bar{\tau}_k \Rightarrow (\tau_B \Rightarrow \tau_A)))=1$ or, equivalently, $h(\bar{\tau}_k) \leq h(\tau_B \Rightarrow \tau_A)$. Now, let  $\mathcal{B}$ be a Boolean algebra and let $v$ be a valuation over the matrix $\mathcal{M}(\mathcal{B})$ such that $v(B) \in \textrm{D}_{\mathcal{B}}$. Let  $h:For_3^k(\Sigma_{BA}) \to \mathcal{B}$ be a homomorphism such that $h(p_i^j)=v(p_i)_j$ for $1 \leq i\leq k$ and $1 \leq j\leq 3$. Then  $h(\bar{\tau}_k)=1$, since $v(p_i) \in \textsc{B}_{\letk}^{\mathcal{B}}$ for $1 \leq i \leq k$. From this, $h(\tau_B \Rightarrow \tau_A)=1$, that is,   $h(\tau_B) \leq h(\tau_A)$. But $h(\tau_B)=1$, given that  $v(B) \in \textrm{D}_{\mathcal{B}}$ (which means that $v(B)_1=1$). From this we conclude that  $h(\tau_A)=1$. This means that $v(A)_1=1$, i.e. $v(A) \in \textrm{D}_{\mathcal{B}}$. This shows that $\Gamma \models_{\mathcal{M}(\mathcal{B})} A$ for every $\mathcal{B}$, hence  $\Gamma \models_{Mat(\letkp)} A$.\\[1mm]
(Right to left - Completeness): Suppose that $\Gamma  \models_{Mat(\letkp)} A$. Then, in particular,  $\Gamma \models_{\mathcal{M}(\mathcal{B}_2)} A$. But $\mathcal{M}(\mathcal{B}_2)$ is $\mathcal{M}_6$, hence $\Gamma \models_{\mathcal{M}_6} A$. By Theorem~ \ref{comple-LETKP-matrix}, $\Gamma \vdash_{\letkp} A$.
\end{proof}\end{theorem}

\begin{remark} \label{lattice}
Recall that, in addition to the usual order-theoretic definition,   a  {\em lattice} can be equivalently defined as an algebra $\langle L, \sqcap,\sqcup\rangle$ such that: 
\begin{itemize} 
\item[] (1)~$a \sqcap a= a = a \sqcup a$; 
\item[] (2)~$a \sqcap b=b \sqcap a$ and $a \sqcup b=b \sqcup a$; 
\item[] (3)~$a \sqcap(b \sqcap c)=(a \sqcap b) \sqcap c$ and $a \sqcup(b \sqcup c)=(a \sqcup b) \sqcup c$; and
\item[] (4) $a \sqcap(a \sqcup b)=a= a \sqcup(a \sqcap b)$, for every $a,b,c \in L$
\end{itemize}
\noi 
and so  the partial order is defined as: $a \leq b$ iff  $a=a \sqcap b$ (iff $b= a \sqcup b$). 
\end{remark}

\begin{theorem}  \label{twist-lattice}
For every Boolean algebra $\mathcal{B}$ the twist structure $\mathcal{T}_{\mathcal{B}}$ is a bounded lattice in which the infimum and supremum are given by $\tilde{\land}$ and $\tilde{\lor}$, respectively, and  $T=(1,0,1)$ and $F=(0,1,1)$ are the top and bottom elements (where $1$ and $0$ are the top and bottom elements of $\mathcal{B}$).
\begin{proof} It will be shown that, for every Boolean algebra $\mathcal{B}$, the algebra $\mathcal{T}_{\mathcal{B}}$ is such that $\tilde{\land}$ and $\tilde{\lor}$ satisfy conditions (1)-(4) of Remark~\ref{lattice}.

\m\mh  
(1) Let $\#\in \{\land,\lor\}$. Given that $\mathcal{B}$ satisfies condition~(1), it is clear that $(z \,\tilde{\#}\, z)_i=z_i$ for $i=1,2$. On the other hand,  $(z \,\tilde{\#}\, z)_3=(z_1 \sqcap z_3) \sqcup (z_2 \sqcap z_3)= z_3 \sqcap (z_1 \sqcup z_2)=z_3$. Hence, $z \,\tilde{\#}\, z=z$ for every $z$ and $\#\in \{\land,\lor\}$.

\m\mh 
(2) Clearly $z \,\tilde{\#}\, w=w\,\tilde{\#}\, z$ for every $z,w$ and  $\#\in \{\land,\lor\}$, by the very definitions and by the 
fact  that $\mathcal{B}$ satisfies condition~(2). 

\m\mh 
(3) Let us first prove that $z \,\tilde{\land}\,(w \,\tilde{\land}\, u)=(z \,\tilde{\land}\, w) \,\tilde{\land}\, u$. Observe that $(z \,\tilde{\land}\,(w \,\tilde{\land}\, u))_i=((z \,\tilde{\land}\, w) \,\tilde{\land}\, u)_i$ for $i=1,2$, by definition of $\tilde{\land}$ and the fact that $\mathcal{B}$ satisfies condition~(3). Now,  let $a \defi (w \,\tilde{\land}\, u)_3=(w_1 \sqcap w_3 \sqcap u_1 \sqcap u_3) \sqcup (w_2 \sqcap w_3) \sqcup (u_2 \sqcap u_3)$ and $b\defi (z \,\tilde{\land}\, w)_3= (z_1 \sqcap z_3 \sqcap w_1 \sqcap w_3) \sqcup (z_2 \sqcap z_3) \sqcup (w_2 \sqcap w_3)$. Hence,

\begin{itemize}\setl 
\item[] $c\defi (z \,\tilde{\land}\,(w \,\tilde{\land}\, u))_3=(z_1 \sqcap z_3 \sqcap (w_1 \sqcap u_1) \sqcap a) \sqcup (z_2 \sqcap z_3) \sqcup ((w_2 \sqcup u_2) \sqcap a)$, and
\item[] $d\defi ((z \,\tilde{\land}\, w) \,\tilde{\land}\, u)_3=((z_1 \sqcap w_1) \sqcap b \sqcap u_1 \sqcap u_3) \sqcup ((z_2 \sqcup w_2) \sqcap b) \sqcup (u_2 \sqcap u_3)$.
\end{itemize}

\noi By using an automatic prover for  tautologies in \cpl\ it is immediate to check that the formula $(A_c \Rightarrow A_d) \sqcap (A_d \Rightarrow A_c)$ is a tautology in \cpl\  (expressed in the signature $\Sigma_{BA}$), where $A_c$ and $A_d$ are the propositional formulas in $For_3^3(\Sigma_{BA})$, respectively obtained from the terms $c$ and $d$ by replacing $z_j, w_j,u_j$ by the propositional variables $p_1^j,p_2^j,p_3^j$, for $1 \leq j \leq 3$ (recalling Remark~\ref{formulas-terms} and the notation established therein). This means that the equation $A_c\approx A_d$ holds in $\mathcal{B}_2$ and so it holds in every Boolean algebra $\mathcal{B}$, as observed in Remark~\ref{val-Boole}(2). That is, $(z \,\tilde{\land}\,(w \,\tilde{\land}\, u))_3=((z \,\tilde{\land}\, w) \,\tilde{\land}\, u)_3$ for every $\mathcal{B}$ and every $z,w,u$ in $\textsc{B}_{\letk}^{\mathcal{B}}$. Therefore, $z \,\tilde{\land}\,(w \,\tilde{\land}\, u)=(z \,\tilde{\land}\, w) \,\tilde{\land}\, u$  for every $\mathcal{B}$ and every $z,w,u \in \textsc{B}_{\letk}^{\mathcal{B}}$. The proof of associativity of $\tilde{\lor}$ is obtained by similar arguments.

\m\mh
(4) Observe that, by the definition of $\tilde{\land}$ and $\tilde{\lor}$, $z_i=(z \,\tilde{\lor}\,(z \,\tilde{\land}\, w))_i$ for every $z,w$ and $i=1,2$, given that $\mathcal{B}$ has property~(4). Now, let  $a\defi (z \,\tilde{\land}\, w)_3= (z_1 \sqcap z_3 \sqcap w_1 \sqcap w_3) \sqcup (z_2 \sqcap z_3) \sqcup (w_2 \sqcap w_3)$ and $b \defi (z \,\tilde{\lor}\,(z \,\tilde{\land}\, w))_3=(z_2 \sqcap z_3 \sqcap(z_2 \sqcup w_2) \sqcap a) \sqcup (z_1 \sqcap z_3) \sqcup(z_1 \sqcap w_1 \sqcap a)$. As we have done in item~(3), let $A_b$ be the propositional formula  in $For_3^2(\Sigma_{BA})$ obtained by replacing $z_j$ and $w_j$ in the expression $b$ by the propositional variables $p_1^j$ and $p_2^j$, for $1 \leq j \leq 3$. By using an automatic prover for  tautologies in \cpl\ it can be checked that, in this case, the formula $A \defi (p_1^3 \Rightarrow A_b) \sqcap (A_b \Rightarrow p_1^3)$ {\em is not} a tautology in \cpl\ (expressed in the signature $\Sigma_{BA}$). However, the  only rows in which $A$ gets the value $0$ is when the triple $p_1=(p_1^1,p_1^2,p_1^3)$ gets the value $(0,0,1)$ or when the triple $p_2=(p_2^1,p_2^2,p_2^3)$ gets the value $(1,1,1)$, and these triples do not correspond to snapshots in $\textsc{B}_{\letk}$. Consider then the formula $B\defi \bar{\tau}_2 \Rightarrow A$ (where $\bar{\tau}_2$ is defined as in Remark~\ref{formulas-terms} by taking $k=2$). By the previous considerations, it follows that $B$ is a tautology. That is, the equation $(\bar{\tau}_2 \Rightarrow A)\approx \top$ in the language of Boolean algebras
holds in $\mathcal{B}_2$. By Remark~\ref{val-Boole}(2), that equation holds in any Boolean algebra $\mathcal{B}$. That is, for every homomorphism $h:For_3^2(\Sigma_{BA}) \to \mathcal{B}$, $h(\bar{\tau}_2 \Rightarrow A)=1$ or, equivalently, $h(\bar{\tau}_2) \leq h(A)$. Given  a Boolean algebra $\mathcal{B}$, let  $z=(z_1,z_2,z_3)$ and $w=(w_1,w_2,w_3)$ in $\textsc{B}_{\letk}^{\mathcal{B}}$, and let $h:For_3^2(\Sigma_{BA}) \to \mathcal{B}$ be a homomorphism such that $h(p_1^j)=z_j$ and  $h(p_2^j)=w_j$ for $1 \leq j \leq 3$. Then, $h(\bar{\tau}_2)=1$ and so $h(A)=1$.  This means that $z_3=b=(z \,\tilde{\lor}\,(z \,\tilde{\land}\, w))_3$, hence $z=z \,\tilde{\lor}\,(z \,\tilde{\land}\, w)$ for every $z,w$  in $\textsc{B}_{\letk}^{\mathcal{B}}$. In a similar way it can be proved that $z=z \,\tilde{\land}\,(z \,\tilde{\lor}\, w)$ for every $z,w$  in $\textsc{B}_{\letk}^{\mathcal{B}}$. 

\m\mh This shows that the twist structure $\mathcal{T}_{\mathcal{B}}$ is a lattice, for every Boolean algebra $\mathcal{B}$. Clearly, $z \,\tilde{\land}\, T= z = z \,\tilde{\lor}\, F$, for every $z \in \textsc{B}_{\letk}^{\mathcal{B}}$. Therefore  $\mathcal{T}_{\mathcal{B}}$ is a bounded lattice with top and bottom elements given by $T$ and $F$, respectively.
\end{proof}\end{theorem}

\begin{proposition} \label{def-order-twist}
The order  in the lattice $\mathcal{T}_{\mathcal{B}}$ is given as follows:
\begin{itemize}
\item[] $(z_1,z_2,z_3) \leq (w_1,w_2,w_3)  $  if and only if:   

\ \quad \quad \quad  $ z_1 \leq w_1, \ z_2 \geq w_2, \ z_2 \sqcap z_3 \geq w_2 \sqcap w_3,   
\mbox{ and } z_3 \leq (z_1 \sqcap w_3) \sqcup z_2$.
\end{itemize}
\begin{proof}
By definition of the order $\leq$ in $\textsc{B}_{\letk}^{\mathcal{B}}$ induced by the algebraic lattice structure of $\mathcal{T}_{\mathcal{B}}$,  and  according to Theorem~\ref{twist-lattice}, it follows that 
$(z_1,z_2,z_3) \leq (w_1,w_2,w_3)$ \ iff  \ $z_1 \leq w_1$, $z_2 \geq w_2$,  \  and \\[1mm]
\indent $(\ast) \ \ \ z_3=(z_1 \sqcap z_3 \sqcap w_3) \sqcup (z_2 \sqcap z_3) \sqcup (w_2 \sqcap w_3)$.\\[1mm] 
By taking infimum w.r.t. $z_2$ in both sides of $(\ast)$ we get that $z_2 \sqcap z_3 = (z_2 \sqcap z_1 \sqcap z_3 \sqcap w_3) \sqcup (z_2 \sqcap z_3) \sqcup (z_2 \sqcap w_2 \sqcap w_3)$. Given that $z_2 \sqcap z_1 \sqcap z_3=0$ and $z_2 \sqcap w_2=w_2$ (since $z_2 \geq w_2$), this implies that $z_2 \sqcap z_3 = (z_2 \sqcap z_3) \sqcup (w_2 \sqcap w_3)$. That is,  $z_2 \sqcap z_3 \geq w_2 \sqcap w_3$. Moreover, by considering the latter relation in equation $(\ast)$ we get that $z_3 = (z_1 \sqcap z_3 \sqcap w_3) \sqcup (z_2 \sqcap z_3) = z_3 \sqcap ((z_1 \sqcap w_3) \sqcup z_2)$.  This means that $ z_3 \leq (z_1 \sqcap w_3) \sqcup z_2$. The proof of the converse is analogous (but a little easier).
\end{proof}\end{proposition}

\subsection{On the lattice structure of \letkp} \label{SectSix}

In Belnap  \citep{belnap.1977.how} we find two lattice-orderings defined by the four semantic values of \fde\ (and 
 by the four values of \fdeto\ as well), called 
\textbf{A4} and \textbf{L4}. 
 The lattice \textbf{A4} has   \textsf{n} at the bottom and \textsf{b} at the top: 

$$
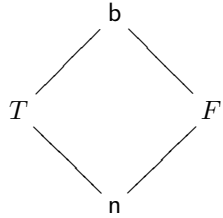
\xymatrix{
&{\bo} \ar@{-}[dl]\ar@{-}[dr] &&\\
T \ar@{-}[dr]   & & F \ar@{-}[dl]  \\
& \nei &
}$$ 
 
 \m\mh
The partial order  $a \leq b$ is read as `$a$ approximates the information in $b$'.   
 The underlying idea is that 
 the amount of information     grows from  bottom to top, 
 in the sense that \textit{T} and \textit{F} convey  more information that \textsf{n}, 
 and \textsf{b} conveys more information than both  \textit{T} and \textit{F}. In Belnap's words,

\begin{quote}
None is at the bottom because it gives no information at all; and Both is at the top because it gives too much (inconsistent) information 
\cite[p.~39]{belnap.1977.how}.  
\end{quote}

\mh  This order  is  clear if we think of the values \textsf{n}, \textit{F}, \textit{T}, and 
 \textsf{b}  as subsets of $\{ 0,1\}$,   respectively, $\emptyset$, $\{ 0\}$, $\{1 \}$, and $\{0,1 \}$, 
 and $\leq$ as the relation of inclusion $\subseteq$.  
Note that:  (i) positive and negative information, represented by the values \textit{T} and \textit{F} assigned to a sentence $A$, 
are on a par in this order, and (ii)  it is    assumed that a contradiction   
 $A\land\neg A$ not only does contain information but also contains the highest amount of information on $A$. 
 This is in line with 
  the already mentioned notion  of information as meaningful data, which considers false information as information 
  (see e.g.      \cite{dunn2008,fetzer2004a}), and is the notion of information that underlies the  
 interpretation of \fde\ worked out by Belnap and Dunn, as well as  the intended interpretation of \lets\ in terms of information.   
 
The \textit{logical lattice} \textbf{L4}, on the other hand,    
 has   \textit{F} at the bottom and \textit{T} at the top. 
 The join operation is given by $\tilde{\lor}$ and the meet by $\tilde{\land}$. 
 It is represented by the following diagram: 

$$
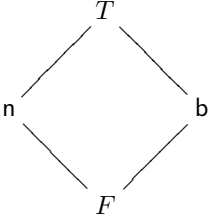
\xymatrix{
&{T} \ar@{-}[dl]\ar@{-}[dr] &&\\
\nei \ar@{-}[dr]   & & \bo \ar@{-}[dl]  \\
& F &
}$$

\mh The partial order of \textbf{L4} can be defined as follows. Think  of the values $T, F, \bo$, and $\nei$ as pairs $(a_1,a_2)$, 
where $a_1$ and $a_2$ represent, respectively, the values of  formulas $A$ and $\neg A$ in a given bivaluation. 
Thus, $T, F, \bo$, and $\nei$ are rep\-re\-sent\-ed, respectively, by the pairs $(1,0), (0,1), (1,1),$ and $(0,0)$. 
Now, the order is given as follows: $(a_1,a_2) \leq ( b_1,b_2)$ if and only if $a_1 \leq b_1 $ and $ a_2 \geq b_2$. 
This order 
is informally explained by Belnap as follows: 
\begin{quote}
[T]he worst thing is to be told something
is false, simpliciter. You are better off (it is one of our hopes) in either being told nothing
about it, or in being told both that it is true and also that it is false; while of course best
of all is to be told it is true, with no muddying of the waters 
\cite[p.~42]{belnap.1977.how}.
\end{quote}
 
Now, let us call \textbf{L6} the lattice obtained by extending 
 \textbf{L4}  to six values. 
The values \textit{T} and \textit{F} of \fde\ become $T_0$  and $F_0$, and we add the values \textit{T} and \textit{F} of \letkp\ 
as, respectively,  a new top  and a new bottom. 
The order of \textbf{L6} has been given in Proposition~\ref {def-order-twist}. The lattice structure of \textbf{L6} can be displayed  as follows:

$$
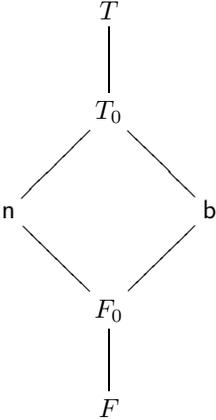
\xymatrix{
&T\ar@{-}[d]&\\
& T_0\ar@{-}[dl]\ar@{-}[dr] &\\
\nei \ar@{-}[dr] & & \bo\ar@{-}[dl]\\
& F_0\ar@{-}[d]\\
&F&
}$$

It should be observed that the six-valued lattice above is exactly the lattice $\mathcal{T}_{\mathcal{B}_2}$
with the order defined as in Proposition~\ref{def-order-twist}.

\begin{remark} \label{crystal}
The lattice structure of  \textbf{L6}, together with its negation (and expanded by a suitable implication), 
has already appeared in the context of relevance logic in Routley  (later Sylvan) \cite[p.~224]{rout:79}. 
 According to this author, this structure (also called {\em crystal lattice}) was first proposed by   Meyer.\footnote{In the context of order theory this lattice was introduced, possibly for the first time in the literature by means of a Haase diagram, in~\cite{klein:35}, fig. 11, p. 613. We thank Rodolfo Ertola for pointing out this fact to us.} As mentioned by  Brady in~\cite[pp. 65-66]{brady.1984}, 
  this six-valued lattice structure together with its negation and implication characterizes, as a logical matrix in which every element other than $F$ is designated, the finitely axiomatizable relevance system {\bf CL} (see Sylvan et al. \cite[p. 114]{sylvan:2003}).
The implication $\to_c$ of the crystal lattice  is such that $A \to_c (B \to_c A)$ is not a valid schema, hence the logic {\bf CL} is different from the $\circ$-free fragment of \letkp.
 For more results about the crystal lattice see~\cite{kramer:2020}.

\end{remark}

It is worth noting that, in the diagram of \textbf{L6} above (including its negation),  \textbf{L4} corresponds to  the inner diamond.
The order of \textbf{L6} can be explained by modifying  the quotation above from Belnap:   
the worst thing is to be told something
has been \textit{conclusively} established as false, which is the value \textit{F};   
not too bad is to be false but not conclusively false, the value $F_0$. 
To be told something is true, although not conclusively, the value $T_0$,  
is better, but the best of all is to be told it is conclusively true, which is the value \textit{T}.

Finally, it is worth mentioning that  if we lay on its side the \textbf{L6} lattice, with \textsf{n} at the bottom,   
we obtain a  meet-semilattice --  call it \textbf{A6} --   that fits the idea of the approximation lattice \textbf{A4}: 
  the amount of information     grows from  bottom to top, but 
  from the nodes $T_0$ and  $F_0$, the new information can be that the information already available is reliable, and in this case
  we get the values $T$ and $F$ respectively, or that conflicting information is obtained, and in this case $T_0$ and  $F_0$ 
   collapse    in the value $\bo$. 
  
  The order of \textbf{A6} can be represented by the relation of inclusion $\subseteq$. Consider the set $\{ 0, 1, \mathsf{c} \}$, 
  where 0 and 1 mean, respectively, negative and positive information, and $\mathsf{c} $, together with 0 or 1, means that the 
  respective information is reliable. The six values \textsf{n}, $T_0$, $F_0$, \textit{T}, \textit{F}, \textsf{b}  correspond, 
  respectively, to the following sets:  $\emptyset$, $\{ 1\}$, $\{ 0 \}$, $\{ 1, \mathsf{c} \}$, $\{ 0, \mathsf{c} \}$, 
   $\{ 1,0 \}$. Note that the sets $\{ 0,1,\mathsf{c} \}$ and $\{ \mathsf{c} \}$ have been dropped from the powerset of 
   $\{ 0, 1, \mathsf{c} \}$: the former because it cannot be that positive and negative information together are reliable, 
   and the latter because $\mathsf{c}$ only makes sense together with either positive or negative information. 
The meets are  given by $\cap$  and the (existing) joins by $\cup$, both operations      
   restricted to the given domain of six subsets of $\{0,1,\mathsf{c}\}$. 
   Taking this into account, $\{ 0, \mathsf{c} \} \cap  \{ 1,\mathsf{c} \} $   is the empty set, rather than $\{ \mathsf{c} \}$, 
   since the latter  does not belong to the domain of the meet-semilattice.
  Analogously, note  that the join of  \textit{T} and \textit{F}, as well as the joins of   \textit{T} and \textsf{b}, and 
  also \textsf{b} and  
  \textit{F},  do  not exist in the meet-semilattice.   
   Then, the least element is  \textsf{n}, but there is no greatest element.

$$
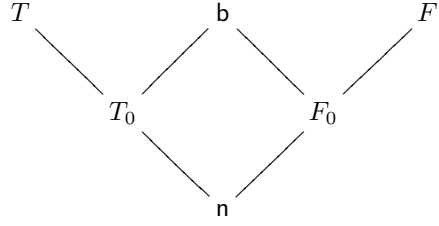
\xymatrix{
T \ar@{-}[dr]   &{}&{\textsf{b}} \ar@{-}[dl]\ar@{-}[dr] & & F\ar@{-}[dl]&\\
&T_0 \ar@{-}[dr]   & & F_0 \ar@{-}[dl]  \\
&& \textsf{n}   &
}$$

\m\mh Note that the diagram above can  be interpreted bottom-up as stages of a database with respect  to a sentence $A$. 
In the bottom there is no information about $A$, neither positive nor negative, which corresponds to the semantic value 
\textsf{n} assigned to $A$.  
From this stage, there are two alternatives: either positive information $A$ or negative information $A$ (i.e., $\neg A$) 
is obtained, that is, $A$ is assigned, respectively, $T_0$ and $F_0$. Now, in each case, two alternatives are possible: 
either the information that  $A$ ($\neg A$) is reliable is obtained, yielding the value $T$ ($F$), or contradictory 
information is obtained, and so $\textsf{b}$ is assigned to $A$ (and to $\neg A$ as well).

\section{Extending a minimal $LET$: the logic \letfp}   \label{sec.letfp}

We have  already mentioned the logic  \letfm, a  minimal logic of evidence and truth that extends \fde\ with the classicality operator  
\con\  and the  rules $EXP^{\circ}$  and $PEM^{\circ}$. 
The logic  \letfp\  is  the extension of \letfm\  with the axiom $\con\con A$ and the rules of propagation of 
classicality for \disj, \conj, and \wneg\ taken from Definition~\ref{propag-rules}.  
 It can also be  defined as the \imp-free fragment of \letkp.

In this brief section we start by \letfm, which admits a valuation semantics and  a (non-deterministic) sound and complete 
six-valued semantics based on swap-structures.  We then move to \letfp, which, like \letkp, is semantically characterized by a six-valued logical matrix, as well as by a class of logical matrices based on twist structures.

Thus, consider the propositional signature  $\Sigma_1=\{\land,\lor, \neg,\cons\}$.
A natural deduction system for  \letfm\   
is obtained    by dropping  rules ${\to} I$, ${\to} E$, $\to_{CL}$, $\neg{\to} I$, and $\neg{\to} E$ 
from Definition~\ref{def.ND.letk}. 
A  bivalued semantics  for \letfm\ is obtained by dropping clauses (v3) and (v7) from Definition~\ref{def-val-LETK}.   
 An Nmatrix  ${\mathcal{M}_{\letfm}}$ for \letfm\ is obtained 
 by dropping clause (iii) of Definition~\ref{defNmatLETK}, which produces the six-valued (non-deterministic)  truth-tables obtained just by dropping  the table for implication of \letk\ displayed in Subsection~\ref{tablesLETKP}.
 As expected,  the syntactic consequence ($ \vdash_{\letfm}  $) and the semantic consequence, defined by 
either the bivalued semantics   ($\models_{\letfm}^2$) or by the six-valued semantics ($ \models_{\mathcal{M}_{\letfm}}  $),  
are equivalent:  
\begin{theorem}  \label{prop.sound.comp.letfm}
$\Gamma\vdash_{\letfm} A$ \ if and only if \ $\Gamma\models_{\letfm}^2 A$ \ if and only if \ $\Gamma\models_{\mathcal{M}_{\letfm}} A$.
\end{theorem}
\mh The  proof of  Theorem~\ref{prop.sound.comp.letfm}  can be easily adapted 
from the proofs  of Theorems \ref{adeq-LETK-bival}, \ref{sound-Nmat-LETK}, and \ref{comp-Nmat-LETK}.   
    The Nmatrix $\mathcal{M}_{\letfm}$, of course, provides a decision procedure for \letfm. 

  \m 
  
We now turn to the logic \letfp. Recall from Remark~\ref{twist-LETKP}  that  $\mathcal{A}_6$ denotes the six-valued algebra  underlying the matrix $\mathcal{M}_6$ of \letkp. Let $\mathcal{A}_6^1$ be the six-valued algebra  underlying the matrix $\mathcal{M}_6^1$ obtained from $\mathcal{M}_6$ by removing the implication operator $\to$. Consider the valuation semantics for \letfp\ obtained from the one for \letkp\ by removing the clauses for implication $\to$. It is easy to see, by adapting the corresponding proofs for \letkp, that
\begin{theorem}  \label{prop.sound.comp.letfp}
$\Gamma\vdash_{\letfp} A$ \ if and only if \ $\Gamma\models_{\letfp}^2 A$ \ if and only if \ $\Gamma\models_{\mathcal{M}_6^1} A$.
\end{theorem}
\mh To generalize: given a twist structure $\mathcal{T}_{\mathcal{B}}$ for \letkp\ induced by a Boolean algebra  $\mathcal{B}$ (recall Definition~\ref{twist-BA}),  let $\mathcal{T}_{\mathcal{B}}^1$ be its implication-free reduct to $\Sigma_1$. Let $\mathcal{M}^1(\mathcal{B})$ be the induced Nmatrix as in the case of \letkp. Clearly,  $\mathcal{T}_{\mathcal{B}_2}^1$ is exactly $\mathcal{A}_6^1$, while $\mathcal{M}^1(\mathcal{B}_2)=\mathcal{M}_6^1$.  Let $Mat(\letfp)$ be the class of logical matrices of the form $\mathcal{M}^1(\mathcal{B})$, and let $\models_{Mat(\letfp)}$ be the associated consequence relation. Then, 
\begin{theorem}  \label{prop.sound.comp.swap.letfp} 
$\Gamma\vdash_{\letfp} A$ \ if and only if \ $\Gamma\models_{Mat(\letfp)} A$. 
\end{theorem}

\subsection{\letfp\ and involutive Stone Algebras} \label{sectISAs}

By convenience, in this subsection we will consider that \letkp\ and \letfp\ are defined over a signature containing the constants $\top$ and $\bot$.

As mentioned in Remark~\ref{crystal}, the lattice structure {\bf L6} expanded by negation (that is, the De Morgan $\{\land,\lor,\neg\}$-reduct of $\mathcal{A}_6$) presented in Subsection~\ref{SectSix} coincides with (the De Morgan reduct of) the so-called Meyer's crystal lattice.

In this subsection it will be shown that the crystal lattice also appears in a different algebraic context. Indeed, a curious and unexpected close relationship between \letfp\ and a variety of algebras known as {\em Involutive Stone Algebras}  (ISAs, for short) can be established. Involutive Stone algebras are De Morgan algebras with an additional unary operator $\nabla$ satisfying some specific equations. 
That is, ISAs are algebras defined over the signature $\{\land,\lor,\neg,\nabla, \bot,\top\}$. The variety of ISAs was introduced by  Cignoli and  de Galego\footnote{Later, M. de Galego changed her name to M. Sagastume.} in~\cite{cig:gal:81} in the context of \L ukasiewicz-Moisil algebras.  
In~\cite{cig:gal:83} they prove that the variety of ISAs is generated by $\mathbb{S}_6$, a 6-element ISA  whose $\{\land,\lor,\neg\}$-reduct coincides with the lattice {\bf L6} (plus negation) of $\mathcal{A}_6$ displayed in Subsection~\ref{SectSix}; that is, (the De Morgan reduct of) the crystal lattice. In  $\mathbb{S}_6$ the $\nabla$ operator is given by $\nabla(a)=T$ if $a \neq F$, and $\nabla(F)=F$. Observe that, in the implication-free reduct $\mathcal{A}_6^1$ of  $\mathcal{A}_6$, $\nabla$ can be defined by means of the formula $\nabla A \defi A \vee  \neg \cons A$. On the other hand, it is clear that the formula $\cons A \defi \neg \nabla A \vee \neg \nabla \neg A$ defines in $\mathbb{S}_6$ the operator $\tilde{\cons}$. This shows that the six-valued algebra $\mathbb{S}_6$ is  equivalent in expressive power to the algebra $\mathcal{A}_6^1$ for \letfp.

 Cant\'u and   M. Figallo studied in~\cite{cant:fig:18} the logic-preserving degrees of truth of the variety of ISAs, which is defined as follows (taking into account that $\mathbb{S}_6$ generates the variety): $\Gamma \models_{\mathbb{S}_6}^\leq A$ iff either $v(A)=1$, for every valuation $v$ over $\mathbb{S}_6$, or there exist $A_1,\ldots,A_n \in \Gamma$ such that $v(A_1) \sqcap \ldots \sqcap v(A_n) \leq v(A)$, for every valuation $v$ over $\mathbb{S}_6$.   
They  prove in their Theorem~5.5 that the logic {\bf Six}  generated by the 4 matrices over $\mathbb{S}_6$ with the 
set of designated  values $\{\nei,T_0,T\}$,  $\{F_0,\nei,T_0,T\}$, $\{T_0,T\}$ and $\{T\}$ coincides with the logic-preserving degrees of truth of the variety of ISAs. On the other hand,  in~\cite{marc:riv:22}, Proposition~4.1, it was shown that  {\bf Six} can be characterized by the logical matrix over $\mathbb{S}_6$ with set of designated  values $\{\bo,T_0,T\}$. Based on this, in~\cite{cant:fig:22} Cant\'u and  M. Figallo apply a general method introduced by Avron and his collaborators to give a cut-free Gentzen system for {\bf Six}, which constitutes a (syntactic) decision procedure for this logic. It is worth noting that, as a consequence of the characterization of {\bf Six} in terms of the matrix induced by $\mathbb{S}_6$ with set of designated  values  $\{\bo,T_0,T\}$,  the following result is obtained straightforwardly:

\begin{proposition} \label{logi-ISA}
Let $\mathbb{I}\mathbb{S}\mathbb{A}$ be the variety of involutive Stone algebras presented in the signature $\{\land,\lor,\neg,\cons, \bot,\top\}$. Then, \letfp\ coincides with the logic-preserving degrees of truth of $\mathbb{I}\mathbb{S}\mathbb{A}$.\footnote{As observed above, \letfp\ was expanded with the  definable constants $\bot$ and $\top$.}
\end{proposition}

\begin{remark} \label{Six}  \ \\ 
(1) It is worth noting that \letkp\ {\em is not} the degree-preserving  expansion of the logic {\bf Six} 
obtained by adding an implication. Indeed, the logic of \letkp\ is not the logic-preserving degrees of truth of  $\mathcal{A}_6$ (the expansion of $\mathbb{S}_6$ with the implication $\tilde{\to}$) given that, for instance, $\neg(A \to B) \vdash_{\letkp} A$, but    $\neg(A \to B) \not\models_{\mathcal{A}_6}^\leq A$, where $\Gamma \models_{\mathcal{A}_6}^\leq A$ iff, either $v(A)=1$, for every valuation $v$ over $\mathcal{A}_6$, or there exist $A_1,\ldots,A_n \in \Gamma$ such that $v(A_1) \sqcap \ldots \sqcap v(A_n) \leq v(A)$, for every valuation $v$ over $\mathcal{A}_6$. In fact, it is enough to consider a valuation $v$ such that $v(A)=\bo$ and $v(B)=\nei$ (which is perfectly possible when $A$ and $B$ are two propositional variables). In this case, $v(\neg(A \to B))=v(A \to B)=\bo \,\tilde{\to}\, \nei=\nei \not\leq \bo= v(A)$, hence   $\neg(A \to B) \not\models_{\mathcal{A}_6}^\leq A$.\\[1mm]
(2) The algebrization of \letkp\ obtained  in Section~\ref{sect-BPalg} shows that there is an additional difference with its implication-free fragment \letfp. Indeed, as shown in~\cite[Proposition~4.2]{marc:riv:22},  the logic  {\bf Six} 
 is not algebraizable, although it is selfextensional. By Proposition~\ref{logi-ISA}, the same holds for \letfp. On the other hand \letkp, despite being algebraizable, it is not selfextensional: indeed, $\neg(A \to B)$ is equivalent to $A \land \neg B$ but $\neg\neg(A \to B)$ is not equivalent to $\neg(A \land \neg B)$. In fact, while the former is equivalent to $A \to B$, the latter is equivalent to $\neg A \vee B$, and clearly these formulas are inequivalent in \letkp.
Despite these differences, it would be interesting to analyze \letkp\ and \letfp\ with relation to  involutive Stone algebras.
\\[1mm]
(3) In~\cite{gomes:etal:2022}  was proposed the study of  expansions of De Morgan algebras by means of an operator $\circ$ which, among several properties, is able to simultaneously recover   explosion  and excluded middle w.r.t. the De Morgan negation (that is, a classicality operator). They show that the degree-preserving logic associated to these algebras coincides, up to language, with {\bf Six} (and so with \letfp, by Proposition~\ref{logi-ISA}). This means that the algebraic conditions required in~\cite{gomes:etal:2022} for the operator $\circ$ turn  out to be equivalent to the propagation conditions required for $\circ$ in \letfp\ (it is worth noting that no implication connective was considered in~\cite{gomes:etal:2022}). Observe that the purpose of~\cite{gomes:etal:2022} is closely related to that of  \lets~\cite{kolkata,letj,letf}, that is, to define logics based on \fde, the logical counterpart of De Morgan algebras, expanded  with a classicality operator; however, in~\cite{gomes:etal:2022} this is done from an algebraic perspective.
\end{remark}

 We close the paper by pointing out that, by Remark~\ref{formulas-terms}, decidability in \letkp\ can be reduced 
 to checking validity in classical propositional  logic \cpl; in particular, the same holds for \letfp. The latter means that the degree-preserving logic of Involutive Stone Algebras can also be decided by means of standard 2-valued truth-tables.

\

\subsection*{Acknowledgements} The authors acknowledge support from the National Council for Scientific and Technological Development (CNPq,~Brazil), research grants 306530/2019-8 and 310037/2021-2.
The second author  also acknowledges support from Minas Gerais State Agency for Research and Development (FAPEMIG, Brazil), research grant APQ-02093-21.

\

\end{document}